\documentclass{amsart}

\usepackage{enumerate, mathtools}
\usepackage{amsmath,amssymb,amscd,bezier}
\usepackage{amsfonts}
\usepackage{mathrsfs}
\usepackage{color}

\renewcommand{\epsilon}{\varepsilon}

\renewcommand{\mathfrak}{\mathscr}

\theoremstyle{plain}
\newtheorem{theorem}{Theorem}[section]
\newtheorem{corollary}{Corollary}[section]
\newtheorem{lemma}{Lemma}[section]
\newtheorem{proposition}{Proposition}[section]

\theoremstyle{remark}
\newtheorem{remark}{Remark}[section]

\newtheorem{example}{Example}[section]

\theoremstyle{definition}

\title[Global existence for semilinear damped
wave equation]{Global existence  of solutions for semilinear wave
equations in Friedmann-Lema\^itre-
Robertson-Walker spacetime}

\author{ Marcelo Rempel Ebert  and Jorge Marques }

\address{Marcelo Rempel Ebert, Departamento de Computa\c{c}\~ao e Matem\'atica, Universidade de S\~ao Paulo, Ribeir\~ao Preto, SP, 14040-901, Brazil,  email{ebert@ffclrp.usp.br.br}}

\address{Jorge Marques, CeBER and FEUC, University of Coimbra, Av. Dias da Silva 165,
3004-512 Coimbra, Portugal, email{jmarques@fe.uc.pt}}

\begin{document}

\begin{abstract}

We consider the nonlinear massless wave equation  belonging to some family of the
Friedmann–Lema\^itre–Robertson–Walker (FLRW) spacetime.  We prove  the global in time small data solutions
 for supercritical powers in the case of decelerating expansion universe.

\end{abstract}



\maketitle

\section{Introduction}\label{sec:1}
In this paper, we prove the global existence  (in time)  of small
data solutions to the Cauchy problem for the  semilinear  wave
equation with scale-invariant damping and decreasing in time  propagation speed
\begin{equation}\label{eq:DPE}
\begin{cases}
u_{tt}(t,x)- (1+t)^{-2\ell}\Delta u(t,x) + \frac{\beta}{1+t} u_t (t,x)=f(u(t,x)), & t\geq0, \ x\in \mathbf{R}^n,  \\
u(0,x)=0 =u_0(x), & x\in \mathbf{R}^n,\\
u_t(0,x)=u_1(x), & x\in \mathbf{R}^n,
\end{cases}
\end{equation}
with $\ell\in (0,1)$ and $\beta>0$.
We assume that $f(u)=|u|^p$ for some $p>1$ or, more in general, $f$ verifies the following local Lipschitz-type condition
\begin{equation}\label{eq:f}
|f(u)-f(v)| \leq C\,|u-v|\,\big(|u|^{p-1}+|v|^{p-1}\big).
\end{equation}
The case $\beta=2$ in \eqref{eq:DPE} is well known as FLRW spacetime model for the decelerating expansion universe, whereas in the particular  case $\ell=\frac23$ \eqref{eq:DPE} is the nonsingular  covariant  massless field in the Einstein- de Sitter
spacetime (see \cite{GY20}).


Let us start with the  state of the art in the case  $\ell=0$. If $\beta\geq \frac53$ for $n=1$,  $\beta\geq 3$ for
$n=2$, or $\beta\geq n+2$ for $n\geq 3$, by assuming data in the
energy spaces with additional regularity $L^1(R^n)$,  the global (in
time) existence result for (\ref{eq:DPE}) was proved in
\cite{DAmmas} for $p> p_{F}(n)\doteq 1 + \frac2n$, the well known
Fujita index \cite{F66}. The  exponent $p_{F}(n)$  is critical for this
model, that is, for $p\leq p_{F}(n)$ and suitable, arbitrarily small
data, there exists no global weak solution \cite{DAL13}. As
conjectured  in \cite{DAL} and \cite{DALR}, if $\beta$ becomes
smaller with respect to the space dimension  $n$, the critical exponent
increase to $\max\{p_{S}(n+\beta), p_{F}(n)\}$,
 where $p_S$ is the Strauss exponent for the semilinear undamped wave equation \cite{GLS}, \cite{S}.
 In \cite{IS} the authors  proved a blow-up result and gave the upper
bound for the lifespan of solutions to (\ref{eq:DPE}) for  $1<p\leq
p_{S}(n+\beta)$ and  $\beta \in [0, \beta_{\star})$, with  $\beta_{\star}=\frac{n^2+n+2}{n+2}$. It is worth
noticing that if  $\beta \in [0, \beta_{\star})$, then $p_{F}(n)<
p_{S}(n+\beta)$ and,  $p_{F}(n)= p_{S}(n+\beta_{\star})$.\\
Recently D'Abbicco (see \cite{DA21} and \cite{DA20}) proved his conjecture in which the critical exponent is equals to $\max\{p_{S}(n+\beta), p_{F}(n)\}$  for $n=1$ and, also proved the global existence of small data solutions for $p> p_{F}(n)$ and $\beta\geq n$ in space dimension $2\leq n\leq 5$.
As far as we know, it is still a open problem to prove global
existence of small data solutions  for $p> p_{F}(n)$ in the cases
  $\beta_{\star}<\beta< n$ for $n\geq 3$ and for $p> p_{S}(n+\beta)$ for $0<\beta<\beta_{\star}$ for  $n\geq 2$.


 For $\ell\in [0,1)$, $\beta \geq 0$ and $n\geq 2$, let $p_{S}(n, \ell, \beta)$ be the positive root of the quadratic equation
\[ \left( n-1+ \frac{\beta-\ell}{1-\ell}\right)p^2- \left( n+1+ \frac{\beta+3\ell}{1-\ell}\right)p-2=0.\]
Recently, in \cite{TW20} and \cite{TW21} the authors have  proved blow-up in a finite time and upper estimates of the lifespan
for solutions to \eqref{eq:DPE} for
\[ 1< p\leq \max\{ p_{F}(n(1-\ell)), p_{S}(n, \ell, \beta) \}.\]
A blow-up result for $\beta=2$ and $\ell\in (0,1)$ in \eqref{eq:DPE} was also proved in \cite{GY20}.\\
It is worth
noticing that if   $p_{F}(n(1-\ell))= p_{S}(n, \ell, \beta_c(n, \ell))$, where
\[
 \beta_c(n, \ell)\doteq\ell + (1-\ell)\left(n+ 1 - \frac2{p_c}\right)= \frac{n^2(1-\ell)^2 + n(1-\ell)(1+2\ell) +2}{2+ n(1-\ell)}. \]
 In particular, if $\beta\geq \beta_c(n, \ell)$, then $
p_{S}(n, \ell, \beta)\leq p_{F}(n(1-\ell)) $.

In \cite{Bui_Reissig}, the authors proposed a classification of
non-effective and effective dissipation, respectively,  for  the
damped wave equation
\[
u_{tt}(t,x) - a^2(t) \Delta u(t,x) + b(t) u_t(t,x)=0 \,
\]
with increasing speed of propagation. The authors derived sharp
estimates for solutions to the Cauchy problem  and, in the  case of
effective dissipation, i.e.,
\[ b(t) \frac{A(t)}{a(t)}\rightarrow \infty, \ as \ t\rightarrow\infty, \qquad A(t)= 1 + \int_0^t a(\tau)\,d\tau,\]
derived global existence (in time) results for the semilinear
problem with power nonlinearities. A similar classification was
introduced in \cite{Ebert_Reissig} in the case $a\in L^1$.
  A natural generalization for the model (\ref{eq:DPE})
is to consider  a positive and decreasing speed of propagation
$a(t)$, with $a\notin L^1$. But in this paper we restrict ourselves
to  the case in which $a$ is a irrational function, since it includes
interesting models by itself, for instance, if $\ell=\frac23$  in
(\ref{eq:DPE}), the considered model  coincides with the
non-singular wave equation in the Einstein de Sitter space-time
(\cite{GKY}, \cite{GYNA2015}).

The main goal in this paper is to prove, under the assumption of small initial data in $L^1(\mathbf{R}^n)\times H^{k-1}(\mathbf{R}^n)$, $k\geq 1$,  the global existence (in time) of solutions to \eqref{eq:DPE} for supercritical powers $p>p_{F}(n(1-\ell)) $,  by supposing that $\beta\geq \beta_c(n, \ell)$. Combine the obtained results in this paper with
the blow-up results derived in  \cite{TW21} we conclude that $p_F(n,\ell)=1+\frac{2}{n(1-\ell)}$  is the critical exponent  for the global in time
existence of solutions for $\beta\geq  \beta_c(n, \ell)$.\\
As far as we know, it is still a open problem to prove global
existence of small data solutions  to \eqref{eq:DPE} for $ p> p_{S}(n, \ell, \beta)$ and
 $0<\beta\leq \beta_c(n, \ell)$.  It is expected that a similar approach to those used for the semilinear free wave equation may be appropriate
to decrease values of $\beta$ and  to overcame some gaps that appear in this paper.

\section{Main results}
\label{sec:2}
 To simplify the writing, from now we consider
\begin{equation}\label{pfujita}
p_c(n, \ell)\doteq p_{F}(n(1-\ell))= 1 +\frac{2}{n(1-\ell)}.
\end{equation}
In the next two theorems, due to the fact that $p_c(n, \ell) \rightarrow \infty$ as $\ell\rightarrow 1$, the choice of the  spaces of solutions  is  related to  fixed  ranges for
$\ell\in [0,1)$ and the   space dimensions $n \geq 2$.   To state our first result, let us define the following parameters
\begin{equation}\label{qbar}
\bar q \doteq  \frac{2(np_c(n, \ell)-1)}{n+1}, \qquad  q_\sharp\doteq \frac{2(n+1)}{n-1}.
\end{equation}

\begin{theorem}\label{highdimensionssharp}
Let $\ell$ be such that
$$
\left\{
\begin{array}{ll}
0 \leq \ell < 1 -\frac{n-1}{2n}, & \mbox{if  $2\leq n \leq 5$} \\
 1- \frac{2(n+1)}{n(n-3)}\leq \ell  < 1 -\frac{n-1}{2n}, & \mbox{if $ 6\leq n \leq 8$}
\end{array}
\right.
$$
and
\[ \beta \geq  \ell + (n+1)(1-\ell)-  \frac{2}{\bar q}  (1-\ell),\]
  with $\bar q\in [p_c(n, \ell),  q_\sharp]$, where $p_c(n, \ell)$, $q_\sharp$ and $\bar q$ are given by \eqref{pfujita} and \eqref{qbar}.
If
\[p_c(n, \ell)<p \leq \frac{4p_c(n,\ell)}{n+3} + 1, \]
then there exists
$\delta>0$ such that for any initial data
\[  u_1  \in \mathcal{D}= L^1(\mathbf{R}^n)\cap L^2(\mathbf{R}^n), \qquad  ||u_1||_\mathcal{D}\leq \delta,\]
there exists a unique weak solution $u\in C([0, \infty),
L^{p_c}(\mathbf{R}^n)\cap L^{q_\sharp}(\mathbf{R}^n))$ to
(\ref{eq:DPE}). Moreover,  the solution satisfies the following
estimates for $p_c\leq q \leq q_\sharp$:\\
If $  \beta > \ell +( n+1)(1-\ell)-\frac{2}{q}(1-\ell) $ then \footnote{Let $f,g : \Omega \subset \mathbf{R}^n \to
\mathbf{R}$ be two functions. From now one we use the notation
$f\lesssim g$ if there exists a constant $C>0$ such that $ f(y)
\leq C g(y)$ for all $y\in \Omega$. }
\begin{equation}\label{ineqnonlinear}
||  u(t, \cdot) ||_{L^q} \lesssim
(1+t)^{-n\left(1-\frac{1}{q} \right)(1-\ell)}|| u_1||_\mathcal{D}
 \end{equation}
whereas  if $ \ell +(n+1)(1-\ell)-\frac{2}{\bar q}(1-\ell)
\leq \beta \leq \ell +(n+1)(1-\ell)-\frac{2}{q}(1-\ell) $,   then
  for any $\epsilon>0$
\begin{equation}\label{ineqnonlinear2}
||  u(t, \cdot) ||_{L^{q}} \lesssim
(1+t)^{[\epsilon -(n-1)\left(\frac12-\frac1{q}\right)](1-\ell)-\frac{ \beta-\ell}{2(1-\ell)}} || u_1||_\mathcal{D}.
\end{equation}
 \end{theorem}

\begin{remark}\label{crucial}
One of the crucial property in the proof of Theorem \ref{highdimensionssharp} is that $r(q)p_c(n, \ell) < q_\sharp$,
for all $p_c(n, \ell)\leq q\leq q_\sharp$, with $\frac1{r(q)}\doteq \frac1{2n}+ \frac{1}{2}+ \frac{1}{nq}$. This condition is satisfied under some condition on $\ell$, namely,
\[r(q_\sharp)p_c< q_\sharp  \Leftrightarrow  \ell < 1- \frac{4}{q_\sharp(n+1)-2(n-1)}=1 -\frac{n-1}{2n}.\]
Since $r(q)\leq r(q_\sharp)$ for all $p_c\leq q\leq q_\sharp$, we also have
\[ r(q)p_c(n, \ell) < q_\sharp \Leftrightarrow \ell < 1 -\frac{n-1}{2n}.\]
In particular, it implies the existence of $\bar q$ satisfying $\bar q< q_\sharp $.\\
For instance, for $\ell\in \left[0, \frac34\right)$ if $n=2$,  and for $\ell\in \left[0, \frac23\right)$ if $n=3$.\\
Moreover,
\[
r(q_\sharp)p_c(n,\ell)< q_\sharp \Longleftrightarrow p_c(n,\ell)  \left ( 1-\frac{r(q_\sharp)}{q_\sharp}  \right ) +1=\frac{4p_c(n,\ell)}{n+3}   +1 < \frac{q_\sharp}{r(q_\sharp)}=\frac{n+3}{n-1}.
\]
\end{remark}
\begin{remark}
Taking into account that $L^1-L^q$ linear estimates in Corollary 2 of \cite{{DA20}} hold only for $\frac{2(n-1)}{n+1}\leq q \leq q_\sharp$,  in the proof of Theorem \ref{highdimensionssharp} we have to assume
\[p_c(n, \ell)\geq \frac{2(n-1)}{n+1}.
\]
Hence a restriction from below in $\ell$ is also  needed, namely, $\ell \geq 1- \frac{2(n+1)}{n(n-3)}$, $n \geq 4$.
This condition is true for  $\ell = 0$ under the assumption $ 2\leq n \leq 5 $ in Theorem \ref{highdimensionssharp}.
\end{remark}
\begin{remark}\label{aboutqbar}
 Let $\frac{n}{r(q)}= \frac12+ \frac{n}{2}+ \frac{1}{q}$.  Condition \eqref{qbar} means that $\bar q$ is defined by $p_c(n, \ell)r(\bar q)= \bar q$.
 In particular, thanks to $(n-1)p_c(n, \ell) \geq 1$ for $n\geq 2$ it holds
 \begin{eqnarray*}
  \frac1{\bar q}-\frac{n}{p_cr(q_{\sharp})}+\frac{n-1}{q_{\sharp}}&=&\frac{n}{p_c r(\bar q)}-\frac{n-1}{\bar q}-\frac{n}{p_cr(q_{\sharp})}+\frac{n-1}{q_{\sharp}} \\
  &=& \frac1{p_c}\left(\frac{n}{r(\bar q)}-\frac{n}{r(q_{\sharp})}\right) - (n-1)\left(\frac{1}{\bar q}-\frac{1}{q_{\sharp}}\right)\\
  &=&\left(\frac1{p_c} - n+1\right)\left(\frac{1}{\bar q}-\frac{1}{q_{\sharp}}\right)\leq 0.
  \end{eqnarray*}
 In the case   $\bar q=p_c$,  Theorem \ref{highdimensionssharp} yields the threshold value
 \[\beta\geq \ell + (1-\ell)\left(n+ 1 - \frac2{p_c}\right)= \beta_c(n, \ell).\]

If $\ell=0$ then $p_c(n,0)=1+
\frac{2}{n}$ and  $\bar q=2$, so we have to assume $\beta\geq n+1 -\frac{2}{\bar q}=n$.  In particular for $\ell=0$ and $n=2$, this condition coincides with the threshold value
$\beta \geq \beta_c(2, 0)=2$.
\end{remark}

%
 In the next result, the novelty is to use higher  regularity $H^k(\mathbf{R}^n), k> \frac{n}{2},$ in order to consider larger values  on the parameter $\ell$ and to relax  the condition in the upper bound for $p$   in Theorem \ref{highdimensionssharp}.
In this way one can also consider values of $\ell \in \left[\frac{n+1}{2n}, 1\right)$ for $n=3,4$, in particular  include
the decreasing speed of propagation $a(t)= (1+t)^{-\frac23}$,  that appears in the well known Einstein de Sitter model for decelerating expanding universe  \cite{GKY}.
\begin{theorem}\label{highdimensions}
Let   $\ell\in \left(1-\frac2{n}, 1\right)$ for  $2\leq n\leq 4$ and $\ell\in \left[\frac12\left( 1+ \sqrt{1-16/n^2}\right), 1\right)$ for $n\geq 5$.
If $\beta\geq \ell + n(1-\ell)(1+\ell)$ and
$p>p_c(n, \ell)$,
 with $p_c(n, \ell)$ given by \eqref{pfujita},
 then there exists
$\delta>0$ such that for any initial data
\[  u_1\in \mathcal{D}= H^{k-1}(\mathbf{R}^n) \,  \cap \,    L^1(\mathbf{R}^n), \qquad  ||u_1||_\mathcal{D}\leq \delta,\]
 there exists a unique energy solution $u\in C([0, \infty),
 H^k(\mathbf{R}^n))$ to
(\ref{eq:DPE}) with $1+\frac{n\ell}{2}\doteq k\leq p_c$, which satisfies the following
estimates
\begin{equation}
||  u (t, \cdot) ||_{L^2} \lesssim
(1+t)^{\frac{n}{2}(\ell-1) } || u_1||_\mathcal{D};
\end{equation}
and
\begin{equation}\label{ineqnonlineari}
||  u (t, \cdot) ||_{\dot H^{k}}
 \lesssim || u_1||_\mathcal{D}
 \begin{cases}
  (1+t)^{(\ell-1)\left(\frac{n}{2} + k\right)}, & \beta> \ell+ n(1-\ell) + 2k(1-\ell)
\\
(1+t)^{\frac{\ell-\beta}{2} }(\ln(e+t))^{\frac{1}2}, & \beta= \ell+ n(1-\ell) + 2k(1-\ell)
\\
(1+t)^{\frac{\ell-\beta}{2} }, & \beta< \ell+ n(1-\ell) + 2k(1-\ell).
\end{cases}
\end{equation}
%
\end{theorem}
%
\begin{remark}
For $ n\geq 2$ we point out that
$$
\ell>1-\frac2n \Longleftrightarrow k>\frac{n}{2} \Longleftrightarrow p_c(n,\ell)>2.
$$
For $n\geq 5$, we have
 \[ 1+\frac{n\ell}{2}\leq p_c \Longrightarrow \ell \geq \frac12\left( 1+ \sqrt{1-16/n^2}\right).\]  \\
The last condition is equivalent to
$$
k  \geq \frac{n+\sqrt{n^2-16}}{4}+ 1
$$
and $\frac{n}{2} <\frac{n+\sqrt{n^2-16}}{4}+ 1 $, $n\geq 5$, so that
$$
\ell \geq \frac12\left( 1+ \sqrt{1-16/n^2}\right) \Longleftrightarrow  k>\frac{n}{2}.
$$

\end{remark}
\begin{example}
  If $\ell=\frac23$, the conclusion of Theorem \ref{highdimensions} holds for $n=2,3, 4$ with  $\beta\geq\frac{1}3\left( 2+  \frac{5n}3\right)$.
\end{example}

\section{Representation of the solution to the linear Cauchy problem}
\label{subsec:3}

Let $s \geq 0$ be a parameter. We need to solve a family of
parameter dependent linear ($f(u)=0$) Cauchy problems corresponding to
(\ref{eq:DPE}):
\begin{equation}
\label{PL}
 \left \{
\begin{array} {l}
u_{tt}(t,x) - (1+t)^{-2\ell} \Delta u(t,x) + \frac{\beta}{1+t} u_t(t,x)=0,  \, \, t \geq s \\
u(s,x)=g_1(s,x) \\
u_t(s,x)=g_2(s,x).
\end{array}
\right.
\end{equation}
We begin by applying Fourier transform to the solution of the
problem (\ref{PL}). We denote the partial Fourier transform of a tempered distribution or of a function
$u: \mathbf{R_{0}^+}\times \mathbf{R^n}
\rightarrow \mathbf{C}$ with respect to $x$,
 by
~$\hat u=\mathfrak{F} u$ or~$\hat u(t,\cdot)=\mathfrak{F} u(t,\cdot)$.  The notation $\mathfrak{F}^{-1}$ denotes the inverse Fourier transform, in the appropriate sense.\\
Following as in \cite{Ebert_Reissig}, we make the change
of variables $\tau=\frac{(1+t)^{1-\ell}}{1-\ell}|\xi|$ and
$v(\tau,s)=\hat{u}(t,s,\xi)$. If $u(t,s,x)$ is the solution of
(\ref{PL}) then $v(\tau,s)$ satisfies
\begin{equation}
\label{FIVP} \left \{
\begin{array} {l}
v^{\prime\prime}(\tau) + \frac{\beta -\ell}{(1-\ell)\tau} v^{\prime}(\tau) + v(\tau) =0 \\
v \left ( \frac{(1+s)^{1-\ell}|\xi|}{1-\ell} \right)  =\hat{g_1}(s,\xi) \\
v^{\prime}\left ( \frac{(1+s)^{1-\ell}|\xi|}{1-\ell} \right )
=\frac{ \hat{g_2}(s,\xi)}{|\xi|}.
\end{array}
\right.
\end{equation}
Moreover, if we are looking for a solution in the product form
$v(\tau,s)=\tau^{\rho} w(\tau,s)$, then $w(\tau,s)$ is a solution of
the Bessel's differential equation of order $\pm \rho$:
\begin{equation}
\label{Bessel_eq} \tau^2 w^{\prime\prime}(\tau) + \tau
w^{\prime}(\tau) + (\tau^2-\rho^2) w(\tau) =0
\end{equation}
where $\rho =\frac{1-\beta}{2(1-\ell)}$. We will use the set of
Hankel functions, $\{ H^+_{\rho}(\tau), H^-_{\rho}(\tau) \}$ to
write the general solution of the ODE (\ref{Bessel_eq}). First,
according to  \cite{Wirth} we introduce an auxiliary
function
\begin{equation}\label{Hankel}
\psi_{j,\gamma,\delta}(t,s,\xi)=|\xi|^j \left|
                         \begin{array}{cc}
                           H^-_{\gamma}\left( \frac{(1+s)^{1-\ell}|\xi|}{1-\ell} \right) & H^-_{\gamma+\delta}\left(\frac{(1+t)^{1-\ell}|\xi|}{1-\ell} \right) \\
                           H^+_{\gamma}\left( \frac{(1+s)^{1-\ell}|\xi|}{1-\ell} \right) &  H^+_{\gamma+\delta}\left(\frac{(1+t)^{1-\ell}|\xi|}{1-\ell} \right) \\
                         \end{array}
                       \right|
\end{equation}
where $j,\gamma,\delta,s$ are real parameters. Since
$H^{\pm}_{\gamma}=J_{\gamma} \pm i Y_{\gamma}$, we can rewrite it in
the form
\begin{equation}\label{Bessel-Inteiro}
\psi_{j,\gamma,\delta}(t,s,\xi)=2i|\xi|^j \left|
                         \begin{array}{cc}
                           J_{\gamma}\left( \frac{(1+s)^{1-\ell}|\xi|}{1-\ell} \right) & J_{\gamma+\delta}\left(\frac{(1+t)^{1-\ell}|\xi|}{1-\ell} \right) \\
                           Y_{\gamma}\left( \frac{(1+s)^{1-\ell}|\xi|}{1-\ell} \right) &  Y_{\gamma+\delta}\left(\frac{(1+t)^{1-\ell}|\xi|}{1-\ell} \right) \\
                         \end{array}
                       \right|
\end{equation}
if $\gamma, \gamma+\delta \in \mathbf{Z}$, or
\begin{equation}\label{Bessel}
\psi_{j,\gamma,\delta}(t,s,\xi) =2i\csc(\gamma \pi)|\xi|^j \left|
                         \begin{array}{cc}
                           J_{-\gamma}\left( \frac{(1+s)^{1-\ell}|\xi|}{1-\ell} \right) & J_{-\gamma-\delta}\left(\frac{(1+t)^{1-\ell}|\xi|}{1-\ell} \right) \\
                           (-1)^{\delta}J_{\gamma}\left( \frac{(1+s)^{1-\ell}|\xi|}{1-\ell} \right) &  J_{\gamma+\delta}\left(\frac{(1+t)^{1-\ell}|\xi|}{1-\ell} \right) \\
                         \end{array}
                       \right|
\end{equation}
if $\gamma, \gamma+\delta\not\in \mathbf{Z}$, where $J_{\gamma},
Y_{\gamma}$ denote the Bessel functions of the first and second
kind, respectively. We then determine the Fourier multipliers and
the first order partial derivatives with respect to $t$ to represent
$\hat{u}$ and $\hat{u_t}$ in the explicit form.

\begin{lemma}(see \cite{E_M})
Let $u(t,s,x)$ be the solution of (\ref{PL}). Then the partial
Fourier transform of $u$ with respect to $x$, $\hat{u}$, is
represented by
\begin{equation}
\label{FSol} \hat{u}(t,s,\xi)= m_0(t,s,\xi)\hat{g_1}(s,\xi) +
m_1(t,s,\xi)\hat{g_2}(s,\xi)
\end{equation}
with Fourier multipliers and the first order partial derivatives
with respect to $t$ given by
\begin{equation}
\label{M1}
\partial^j_t m_k=\frac{(-1)^{k}\pi
i}{4(1-\ell)}(1+s)^{1+(\beta-1)/2} (1+t)^{(1-\beta)/2-j\ell}
\psi_{1+j-k,\rho+k-1,1-j-k}
\end{equation}
where $\rho=\frac{1-\beta}{2(1-\ell)} $, $k,j=0,1$.
\end{lemma}

\section{ $L^p-L^q$  estimates}
\label{subsec:4}

In order to obtain an estimate of (\ref{FSol}) we have to
distinguish between large and small $\tau$ values. We divide the
extended phase space $\mathbf{R_{0}^+} \times \mathbf{R_{0}^+}
\times \mathbf{R^{+}}$ into three zones. We
define the zone of high frequencies
$$
Z_1=\{(t,s,|\xi|): |\xi| \geq  (1+s)^{\ell-1} \} ;
$$
and the zones of low frequencies
$$
Z_2=\{(t,s,|\xi|): (1+t)^{\ell-1} \leq |\xi| \leq
 (1+s)^{\ell-1} \} ;
$$
$$
Z_3=\{(t,s,|\xi): |\xi| \leq  (1+t)^{\ell-1} \} ,
$$
separated by the boundary $\{(t,|\xi|): (1+t)^{1-\ell} |\xi| = N
(1-\ell) \}$.\\
We consider the cut-off function $\chi \in C^{\infty}(\mathbb{R}^n)$  with
$\chi(r)=1$ for $r\leq \frac12$ and $\chi(r)=0$ for $r\geq 1$ and define
\begin{eqnarray*}
&\chi_1(s,\xi)=1- \chi((1+s)^{1-\ell}|\xi|), \\
&\chi_2(t, s,\xi)= \chi((1+s)^{1-\ell}|\xi|)\left( 1 - \chi((1+t)^{1-\ell}|\xi|)\right) , \\
&\chi_3(t, s,\xi)= \chi((1+s)^{1-\ell}|\xi|) \chi((1+t)^{1-\ell}|\xi|),
\end{eqnarray*}
such that $ \chi_1+\chi_2+\chi_3=1$.

\begin{lemma}{}
\label{Lem1} Let $\ell \in (0,1)$, $\gamma\neq 0$,  and $k \geq 0$. It holds
\begin{eqnarray}\label{psi2}
& & |\xi|^k | \psi_{0,\gamma,0}(t,s,\xi) | \lesssim \\
& \lesssim & \left\{
\begin{array}{ll}
|\xi|^{k-1} (1+s)^{(\ell-1)/2} (1+t)^{(\ell-1)/2} & \mbox{if  $ (t,s,\xi) \in Z_1$} \\

|\xi|^{k-|\gamma|-1/2} (1+s)^{(\ell-1)|\gamma|} (1+t)^{(\ell-1)/2} & \mbox{if  $ (t,s,\xi) \in Z_2$} \\

|\xi|^k(1+s)^{(\ell-1)|\gamma|}(1+t)^{(1-\ell)|\gamma|}
& \mbox{if $ (t,s,\xi) \in Z_3$.}
\end{array}
\right. \nonumber
\end{eqnarray}
for all $s \geq 0$ and $t\geq s$.
\end{lemma}

\begin{proof}{}
For any $ N \in (0,1)$, the following  properties hold:
\begin{eqnarray}
\label{HankelHigh}
 |H^{\pm}_{\gamma}(\tau)| & \lesssim &   \tau^{-\frac12}, \, \tau\in [N, \infty) ;\\
\label{HankelLow} |H^{\pm}_{\gamma}(\tau)| & \lesssim &
\tau^{-|\gamma|}, \, \tau\in (0, N), \, \gamma\neq 0 ;\\
\label{Besselest}
 |J_{\gamma}(\tau)| & \lesssim & \tau^{\gamma}, \, \tau\in (0, N);\\
\label{SecondBesselest} |Y_{\gamma}(\tau)| & \lesssim &
\tau^{-\gamma}, \, \tau\in (0, N), \, \gamma\neq 0.
\end{eqnarray}
To conclude the estimates in zones $Z_1$ and $Z_2$ we may use the
representation (\ref{Hankel}),  estimates (\ref{HankelHigh}) and
(\ref{HankelLow}), whereas in the zone $Z_3$ we use
(\ref{Bessel-Inteiro})-(\ref{Bessel}) and
(\ref{Besselest})-(\ref{SecondBesselest}).
\end{proof}

\begin{proposition}\label{Lq}
Let $ 0 \leq \ell < 1$, $ \beta > 1$, $k\geq 0$, $n \in \mathbb{N} $ and $q\geq 2$.
Assume that $g_1 \doteq 0$ and
$g_2 \in  L^1(\mathbf{R^n})\cap \dot H^{[k-1]_+}(\mathbf{R}^n) \cap L^m(\mathbf{R^n})$, with
 $m\in [1,2]$ such that
 \[m=m(k, n, q )> \frac{nq}{n+q(1-k)}, \quad k\in [0,1).\]
  Then the solution $u$
of the problem (\ref{PL}) satisfies the following \emph{a priori}
estimates:
\begin{description}
\item[(i)]
If  $1< \beta < \ell + 2n(1-\ell)\left( 1-\frac1q\right)+2k(1-\ell)$
then \begin{equation}
 |||D|^k u(t,s,\cdot) ||_{L^q}   \lesssim
(1+t)^{\frac{\ell-\beta}{2}} (1+s)^{1+ \frac{\beta-\ell}{2} +(\ell-1)\left( n\left(1-\frac1q\right)+k\right)}
\left(|| g_2 ||_{L^1} +  (1+s)^{ n(1-\ell)\left(1-\frac1m\right)}|| g_2 ||_{L^m}\right)
  \end{equation}
  for $k\in [0,1)$ and $ 2\leq q<\frac{nm}{[n-m+mk]_+}$, whereas for $k\geq 1$
   \begin{equation}
 || u(t,s,\cdot) ||_{\dot H^{k}}   \lesssim
(1+t)^{\frac{\ell-\beta}{2}} (1+s)^{1+ \frac{\beta-\ell}{2} +(\ell-1)\left( \frac{n}{2}+k\right)}
\left( || g_2 ||_{L^1} + (1+s)^{(1-\ell)\left( \frac{n}{2}+k-1\right)}  \| g_2\|_{\dot H^{k-1}}\right);
  \end{equation}
\item[(ii)] If  $ \beta = \ell + 2n(1-\ell)\left( 1-\frac1q\right)+2k(1-\ell) $
then
   \begin{equation}
 || |D|^ku(t,s,\cdot) ||_{L^q}   \lesssim
(1+s)(1+t)^{(\ell-1)\left( n\left(1-\frac1q\right)+k\right)}
\left( \left(\ln{\left ( \frac{e+t}{e+s} \right )}\right)^{1-\frac1q}|| g_2 ||_{L^1} +  (1+s)^{ n(1-\ell)\left(1-\frac1m\right)} || g_2 ||_{L^m}\right)
 \end{equation}
  for  $k\in [0,1)$ and $ 2\leq q<\frac{nm}{[n-m+mk]_+}$,
 whereas for $k\geq 1$
   \begin{equation}
|| u(t,s,\cdot) ||_{\dot H^{k}}  \lesssim
(1+s)(1+t)^{(\ell-1)\left(\frac{n}{2}+k\right)}
\left( \left(\ln{\left ( \frac{e+t}{e+s} \right )}\right)^{\frac12}|| g_2 ||_{L^1} + (1+s)^{(1-\ell)\left( \frac{n}{2}+k-1\right)}  || g_2 ||_{\dot H^{k-1}}\right);
 \end{equation}
\item[(iii)]
If  $ \beta > \ell + 2n(1-\ell)\left( 1-\frac1q\right)+2k(1-\ell) $
then
   \begin{equation}
 || |D|^ku(t,s,\cdot) ||_{L^q}   \lesssim
(1+s)(1+t)^{(\ell-1)\left( n\left(1-\frac1q\right)+k\right)}
\left(|| g_2 ||_{L^1} + (1+s)^{ n(1-\ell)\left(1-\frac1m\right)} || g_2 ||_{L^m}\right)
 \end{equation}
  for $k\in [0,1)$ and $ 2\leq q<\frac{nm}{(n-m+mk)_+}$ and for $k\geq 1$
   \begin{equation}
 ||  u(t,s,\cdot) ||_{\dot H^{k}}   \lesssim
(1+s)(1+t)^{(\ell-1)\left(\frac{n}{2}+k\right)}
\left( || g_2 ||_{L^1} + (1+s)^{(1-\ell)\left( \frac{n}{2}+k-1\right)}  \| g_2\|_{\dot H^{k-1}}\right);
 \end{equation}
  \end{description}
\end{proposition}
\begin{proof}
Let $k\geq 0$, $q\geq 2$, $\ell \in (0,1)$ and $
\beta > 1 $.\\
Considerations in $Z_{3}$: In the zone $Z_3$, from
Lemma \ref{Lem1} we may estimate
\[
|\xi|^k  | m_1(t,s,\xi) | \lesssim |\xi|^k(1+s).\] By using
Haussdorff-Young inequality and H\"{o}lder inequality, setting
\[ \frac1r = 1-\frac1q, \]
for $q\geq 2$, one may estimate
\begin{align*}
\| \mathfrak{F}^{-1}(\chi_3(s,\xi) |\xi|^k m_1(t,s,\xi))\ast
g_2\|_{L^q}
    & \lesssim \| \chi_3(s,\xi) |\xi|^k m_1(t,s,\xi)\hat g_2\|_{L^{q'}}\\
     &\lesssim \| \chi_3(s,\xi) |\xi|^k m_1(t,s,\xi)\|_{L^r}\|\hat g_2\|_{L^{\infty}} \\
         & \lesssim\, (1+s) (1+t)^{(\ell-1)\left( n\left(1-\frac1q\right)+k\right)} \|g_2\|_{L^1}
             \end{align*}
thanks to
    \[ \| \chi_3(s,\xi) |\xi|^k \|^r_{L^r(Z_3)}=
     \int_{Z_3} |\xi|^{rk}\ d\xi
    \lesssim (1+t)^{(kr+n)(\ell-1)}.
     \]
If   $ \beta <  \ell + 2n(1-\ell)\left( 1-\frac1q\right)+2k(1-\ell) $,  by using that
   \begin{equation}
\label{betapequeno}
  (1+s)(1+t)^{(\ell-1)\left( n\left(1-\frac1q\right)+k\right)}\leq (1+t)^{\frac{\ell-\beta}{2}} (1+s)^{1+ \frac{\beta-\ell}{2} +(\ell-1)\left( n\left(1-\frac1q\right)+k\right)}
  \end{equation}
we obtain
$$
\| \mathfrak{F}^{-1}(\chi_3(s,\xi) |\xi|^k m_1(t,s,\xi))\ast
g_2\|_{L^q} \lesssim\, (1+t)^{\frac{\ell-\beta}{2}} (1+s)^{1+ \frac{\beta-\ell}{2} +(\ell-1)\left( n\left(1-\frac1q\right)+k\right)} \|g_2\|_{L^1}.
$$

     Considerations in $Z_{1}$: In the zone $Z_1$ from
Lemma \ref{Lem1} we may estimate
$$ |\xi|^k| m_1(t,s,\xi) | \lesssim
|\xi|^{k-1}(1+s)^{(\beta+\ell)/2}(1+t)^{(\ell-\beta)/2}.
$$
By using Haussdorff-Young inequality and H\"{o}lder inequality,
setting
\[ \frac1r = \frac1{q'}-\frac1{m'}= \frac1m-\frac1q, \quad m\in [1,2)\]
for   $ 2\leq q <\frac{nm}{(n-m+mk)_+} $ and $k\in [0,1)$, one may estimate
\begin{align*}
&\|\mathfrak{F}^{-1}(\chi_1(s,\xi) |\xi|^km_1(t,s,\xi))\ast
g_2\|_{L^q}
     \lesssim \| \chi_1(s,\xi)  |\xi|^km_1(t,s,\xi)\hat g_2\|_{L^{q'}}\\
     &\lesssim \| \chi_1(s,\xi)  |\xi|^k m_1(t,s,\xi)\|_{L^r}\|\hat g_2\|_{L^{m'}} \\
         & \lesssim\, (1+t)^{\frac{\ell-\beta}{2}}
        (1+s)^{\frac{\ell+\beta}{2}}
(1+s)^{ \left(n \left( \frac1m-\frac1q \right)+k-1 \right) (\ell-1)} \|g_2\|_{L^m}
  \end{align*}
  thanks to
$$ \|  \chi_1(s,\xi)|\xi|^k
\|^r_{L^r(Z_1)} = \int_{|\xi|\geq (1+s)^{\ell-1}} |\xi|^{r(k-1)}
d\xi \lesssim
(1+s)^{\left(n+r(k-1)\right)(\ell-1)}, \quad r(k-1)+n<0.
$$
For $k\geq 1$ we get
 \[ \|\mathfrak{F}^{-1}(\chi_1(s,\xi) |\xi|^k m_1(t,s,\xi))\ast
g_2\|_{L^2} \lesssim  (1+s)^{\frac{\ell+\beta}{2}}
 (1+t)^{\frac{\ell-\beta}{2}}  \| g_2\|_{\dot H^{k-1}}.
\]
If   $ \beta \geq \ell + 2n(1-\ell)\left( 1-\frac1q\right)+2k(1-\ell) $,  by using that
\begin{equation}
\label{betagrande}
      (1+s)^{\frac{\ell+\beta}{2}}(1+t)^{\frac{\ell-\beta}{2}} \lesssim   (1+t)^{(\ell-1)\left( n\left(1-\frac1q\right)+k\right)}  (1+s)^{\ell+(1-\ell)\left( n\left(1-\frac1q\right)+k\right)}
\end{equation}
we may gluing the estimates in zones $Z_1$ and $Z_3$,
  namely,  for $ 2\leq q <\frac{nm}{(n-m+mk)_+} $ and $k\in [0,1)$, we get
\begin{align*}
\|\mathfrak{F}^{-1}(\chi_1(s,\xi) |\xi|^km_1(t,s,\xi))\ast
g_2\|_{L^q}
     \lesssim
      (1+t)^{(\ell-1)\left( n\left(1-\frac1q\right)+k\right)}  (1+s)^{1+  n(1-\ell)\left(1-\frac1m\right)} \|g_2\|_{L^m},
  \end{align*}
  whereas  for $k\geq 1$ we get
  \begin{align*}
\|\mathfrak{F}^{-1}(\chi_1(s,\xi) |\xi|^km_1(t,s,\xi))\ast
g_2\|_{L^2}
     \lesssim
       (1+t)^{(\ell-1)\left( n\left(1-\frac1q\right)+k\right)}  (1+s)^{\ell+(1-\ell)\left( n\left(1-\frac1q\right)+k\right)}  \| g_2\|_{\dot H^{k-1}}.
  \end{align*}
Considerations in $Z_{2}$: In the zone $Z_2$,
from Lemma \ref{Lem1} we may estimate
     $$
|\xi|^k  | m_1(t,s,\xi) | \lesssim \
|\xi|^{k-\alpha}(1+s)(1+t)^{(\ell-\beta)/2},
$$
where $\alpha=\frac{\beta-\ell}{2(1-\ell)}$.
Setting
\[ \frac1r = 1-\frac1q, \]
for $q\geq 2$ then thanks to
    \begin{align*}
    &\| \chi_2(s,\xi) |\xi|^{k-\alpha} \|^r_{L^r(Z_2)}=
     \int_{Z_2}  |\xi|^{(k-\alpha)r}  d\xi
   \lesssim \left\{
\begin{array}{lll}
(1+s)^{(\ell-1)\left(n + r(k-\alpha)\right)} & \mbox{if $ \alpha <
k+n\left(1-\frac1q\right) $} \\
\ln{\left ( \frac{e+t}{e+s} \right )} & \mbox{if $ \alpha =
k+n\left(1-\frac1q\right) $} \\
(1+t)^{(\ell-1)\left(n + r(k-\alpha)\right)} & \mbox{if $ \alpha >
k+n\left(1-\frac1q\right) $}
\end{array}
\right.
\end{align*}
one may estimate
\begin{align*}
&\| \mathfrak{F}^{-1}(\chi_2(s,\xi) |\xi|^k m_1(t,s,\xi))\ast
g_2\|_{L^q}
     \lesssim \| \chi_2(s,\xi) |\xi|^k m_1(t,s,\xi)\hat g_2\|_{L^{q'}}\\
     &\lesssim \| \chi_2(s,\xi) |\xi|^k m_1(t,s,\xi)\|_{L^r}\|\hat g_2\|_{L^{\infty}} \\
     &\lesssim  (1+s)\|g_2\|_{L^1}\left\{
\begin{array}{lll}
(1+t)^{(\ell-\beta)/2}
(1+s)^{(\ell-1)\left( n\left(1-\frac1q\right)+k\right)+(\beta-\ell)/2} & \mbox{if    $ 1 < \beta <\ell + 2n(1-\ell)\left(1-\frac1q\right)+2k(1-\ell) $} \\
(1+t)^{(\ell-\beta)/2}  \left(\ln{\left ( \frac{e+t}{e+s} \right )}\right)^{1-\frac1q} & \mbox{if  $ \beta = \ell + 2n(1-\ell)\left(1-\frac1q\right)+2k(1-\ell) $}\\
(1+t)^{(\ell-1)\left( n\left(1-\frac1q\right)+k\right)} & \mbox{if $
\beta > \ell + 2n(1-\ell)\left(1-\frac1q\right)+2k(1-\ell) $}
\end{array}
\right.
             \end{align*}
\end{proof}

\section{Global existence results}\label{GESD}
\label{subsec:6}

By Duhamel's principle, a function $u\in X$, where~$X$ is a suitable
space, is a solution to~(\ref{eq:DPE}) if, and only if, it satisfies
the equality
\begin{equation}\label{eq:fixedpoint}
u(t,x) =  u^{0}(t, x) + \int_0^t K(t, s, x) \ast_{(x)} \,
f(u(s,x))\, ds\,, \qquad \text{in~$X$,}
\end{equation}
 where $u^0(t,x)$ is the solution to the linear Cauchy problem
\begin{equation}
\label{P}
 \left \{
\begin{array} {l}
u_{tt}(t,x) - (1+t)^{-2\ell} \Delta u(t,x) + \frac{\beta}{1+t} u_t(t,x)=0, \, \, t \geq 0  \\
u(0, x)=0 \\
u_t(0,x)=u_1(x)
\end{array}
\right.
\end{equation}
and  $K(t, s, x) \ast_{(x)} f(u(s,x))$
is the solution to the linear Cauchy problem
(\ref{PL}) with $g_1\equiv 0$ and $g_2\equiv f(u)$, being  $K(t,s,x)=\mathfrak{F}^{-1} (m_1)(t,s,x)$,  i.e.
\[
K(t,s,x)= -\frac{\pi i}{4(1-\ell)}(1+s)^{1+(\beta-1)/2}
(1+t)^{(1-\beta)/2} \mathfrak{F}^{-1} \left (\psi_{0,\rho,0}\right ) (t,s,x) .
\]
The proof of our global existence results is based on the following
scheme: We  define an appropriate data function space $\mathcal{D}$
 and an evolution space for solutions X(T)
%
%
equipped with a  norm relate to the estimates  of
solutions to the linear problem \eqref{P} such that
\[
\|u^{0}\|_{X} \leq C\,\| u_1 \|_{\mathcal D}.
\]

%
For any~$u\in X$, we define the operator~$P$ by
%
\[
P: u\in X(T)\rightarrow Pu(t,x) := u^{0}(t, x) + Fu(t,x),
\]
with
\[
Fu(t, x) \doteq \int_0^t K(t, s, x) \ast_{(x)} \, f(u(s,x))\,
ds\,,
\]
%
then we prove the estimates
\begin{align*}
\|Pu\|_{X}
    & \leq C\,\| u_1\|_{\mathcal{D}}+ C_1(t)\|u\|_{X}^p\,, \\
\|Pu-Pv\|_{X}
    & \leq C_2(t)\|u-v\|_{X} \bigl(\|u\|_{X}^{p-1}+\|v\|_{X}^{p-1}\bigr)\,.
\end{align*}
The estimates for the image $Pu$ allow us to apply Banach's fixed
point theorem. In this way we get simultaneously a unique solution
to $Pu=u$ locally in time for large data and globally in time for
small data. To prove the local (in time) existence we use that
$C_1(t), C_2(t)$ tend to zero as $t$ goes to zero, whereas to prove
the global (in time) existence we use $C_1(t)\leq C$ and $C_2(t)\leq
C$ for all $t\geq 0$.

\subsection{ Proof of Theorem \ref{highdimensions}}

\begin{proof}(Theorem \ref{highdimensions})
   We  define the space

\[
\label{eq:Xsp} X(T)
     \doteq  C([0, \infty),
 H^k(\mathbf{R}^n)),  \quad   \frac{n\ell}2+1 = k\leq p_c(n,\ell)
\]
equipped with the norm
\begin{eqnarray*}
\label{xnorm}
\|u\|_{X(T)}\!\doteq\!\left\{ \begin{array}{lll}
\!\!\displaystyle\sup_{t\in[0,T]} (1+t)^{(1-\ell)\frac{n}{2}}\left(\|  u(t,\cdot)\|_{L^2}+(1+t)^{(1-\ell)k} \|  u(t,\cdot)\|_{\dot H^{k}}\right), \quad \bar k > k,\\
 \!\!\displaystyle\sup_{t\in[0,T]}  (1+t)^{(1-\ell)\frac{n}{2}}\left( \|  u(t,\cdot)\|_{L^2}+(1+t)^{(1-\ell)\bar k}(\ln(e+t))^{-\frac12} \|  u(t,\cdot)\|_{\dot H ^{\bar k}}\right), \quad \bar k =k,\\
\!\!\!\displaystyle\sup_{t\in[0,T]}\!\! \left( (1+t)^{(1-\ell)\frac{n}{2}}
  \| u(t,\cdot)\|_{L^2}\!+\!(1+t)^{\frac{\beta-\ell}{2}}\!\|u(t,\cdot)\|_{\dot H ^{k}} \right),    \frac{n\ell}{2} \leq \bar k < k,
\end{array}
\right.
\end{eqnarray*}
where $\bar k\doteq  \frac{\beta-\ell}{2(1-\ell)}- \frac{n}{2}$.\\
 We have to prove the global existence in time of the solution $u$ assuming that there exists $\delta > 0$ such that
$$
u_1\in \mathcal{D} \doteq H^{k-1}(\mathbf{R}^n) \,  \cap \,    L^1(\mathbf{R}^n),  \qquad  || u_1||_\mathcal{D}\leq \delta.
$$
Thanks to Proposition
\ref{Lq},  $u^{0} \in X(T)$ and it satisfies
%
\[
\|u^{0}\|_{X} \leq C\, \| u_1\|_{\mathcal{D}}.
\]
It remains to show the estimates
\begin{align}
\label{eq:well_th2} \|Fu\|_{X}
    & \leq C\|u\|_{X}^p\,, \\
\label{eq:contraction} \|Fu-Fv\|_{X}
    & \leq C\|u-v\|_{X} \bigl(\|u\|_{X}^{p-1}+\|v\|_{X}^{p-1}\bigr)\,.
\end{align}
Let us begin by prove (\ref{eq:well_th2}). Taking into account the definition of the norm  in the function space $X(T)$, we split the proof accordingly to size of $\beta$:

\begin{itemize}
\item The case $\bar k > k$,  i.e., $\beta> \ell+ n(1-\ell) + 2k(1-\ell)$:

Applying Proposition \ref{Lq}  we
have
\[
  \|  Fu(t, \cdot)\|_{L^2}\lesssim \int_0^t (1+s) (1+t)^{(\ell-1)\frac{n}{2}} \left(\| |u(s, \cdot)|^p\|_{L^1}+ (1+s)^{(1-\ell)n\left( 1-\frac{1}{m}\right)}\| |u(s, \cdot)|^p\|_{L^{m}}\right) ds\]
and
  \[
  \|  Fu(t, \cdot)\|_{\dot H^{k}}\lesssim \int_0^t (1+s) (1+t)^{(\ell-1)\left(\frac{n}{2}+k\right)} \left(\| |u(s, \cdot)|^p\|_{L^1}+ (1+s)^{(1-\ell)\left( \frac{n}{2}+k-1\right)}\| |u(s, \cdot)|^p\|_{\dot H^{k-1}}\right) ds.\]
  %
First,   we use  Gagliardo-Nirenberg inequality
 \begin{eqnarray*}
\| u(s, \cdot)\|_{L^q} \lesssim  \| u(s, \cdot)\|_{L^2}^{1-\theta}\|  u(s, \cdot)\|_{\dot H^{k}}^{\theta}, \quad \theta=\frac{n}{k}\left( \frac12 - \frac1{q}\right),
\end{eqnarray*}
by taking $q=p$ and $q=mp$, where $m\in (1,2]$ such that $ m>\frac{2n}{n+2}$.
Since $u\in X(T)$ we may estimate
\begin{eqnarray*}
\| |u(s, \cdot)|^p\|_{L^{\frac{q}{p}}}& =&\| u(s, \cdot)\|_{L^{q}}^p\lesssim \|  u(s, \cdot)\|_{\dot H^{k}}^{p\theta} \| u(s, \cdot)\|_{L^2}^{(1-\theta)p} \\
&\lesssim&(1+s)^{(\ell-1)\left(\frac{n}{2} + k\right)\theta p+(\ell-1)\frac{n(1-\theta)p}{2} }
\|u\|_{X(T)}^p\lesssim(1+s)^{n(\ell-1)(p-\frac{p}{q})}
\|u\|_{X(T)}^p,
\end{eqnarray*}
for all $p\geq 2$, $q=p$ and $q=mp$.
       Therefore,  we obtain
   \begin{eqnarray*}
  \| Fu(t, \cdot)\|_{L^2}
 & \lesssim &(1+t)^{(\ell-1)\frac{n}{2}} \int_0^t (1+s)^{1 + n(\ell-1)\left(p-1\right)} ds \| u\|_{X(T)}^p  \\
 &+& (1+t)^{(\ell-1)\frac{n}{2}} \int_0^t (1+s)^{1+(1-\ell)n\left( 1-\frac{1}{m}\right)+n(\ell-1)\left(p-\frac{1}{m}\right)}  ds \| u\|_{X(T)}^p\\
 &\lesssim & (1+t)^{(\ell-1)\frac{n}{2}}\| u\|_{X(T)}^p,\end{eqnarray*}
  for    $p> 1+ \frac{2}{n(1-\ell)}.$\\
 Then,  in order to estimate  $\|  Fu(t, \cdot)\|_{\dot H^{k}}$,
 we may  use that $H^k(\mathbf{R^n})$, with $ k>\frac{n}{2}$, is  imbedded   into $L^\infty(\mathbf{R^n})$. Indeed,
  thanks to Corollary \ref{corollaryfractionalpowers},    for
  $p > \max\{1, k-1\}$ we may estimate
 \[\| |u(s, \cdot)|^p\|_{{\dot H}^{k-1}}\le C \| u(s, \cdot)\|_{{\dot H}^{k-1}}\|u(s, \cdot)\|_{L^\infty}^{p-1}.\]
Since $u\in X(T)$ we have
 $$
\| u(s, \cdot)\|_{{\dot H}^{k-1}} \lesssim (1+s)^{(\ell-1)\left(\frac{n}{2}+k-1\right)}\|
u\|_{X(T)},
$$
and
 thanks to Lemma \ref{lem:Sobolev} for  $ \tilde k <\frac{n}{2}<k$ it follows
\[ \|u(s, \cdot)\|_{L^\infty}\lesssim \| u(s, \cdot)\|_{\dot{H}^{\tilde k}} + \| u(s, \cdot)\|_{\dot{H}^{k}} \lesssim (1+s)^{(\ell-1)\left(\frac{n}{2}+\tilde k\right)}
\|u\|_{X(T)}.\]
  If we choose $\tilde k=\frac{n}{2}-\epsilon_0$, with $\epsilon_0$ sufficiently small, then
  \[ \| |u(s, \cdot)|^p\|_{{\dot H}^{k-1}} \lesssim
  (1+s)^{(\ell-1)\left(\frac{n}{2}+k-1\right)+(\ell-1)(n-\epsilon_0)(p-1)} \|u\|_{X(T)}^p, \]
hence
    \begin{eqnarray*}
  \|  Fu(t, \cdot)\|_{\dot H^{k}}&\lesssim & (1+t)^{(\ell-1)\left(\frac{n}{2}+k\right)}\int_0^t (1+s)^{1 + (n-\epsilon_0)(\ell-1)\left(p-1\right)} ds \| u\|_{X(T)}^p \\
  &\lesssim &  (1+t)^{(\ell-1)\left(\frac{n}{2}+k\right)}\| u\|_{X(T)}^p,\end{eqnarray*}
  for    $p> 1+ \frac{2}{n(1-\ell)}.$\\

\item  The case $\frac{n\ell}{2} \leq \bar k < k$,  i.e., $ \ell + n(1-\ell)(1+\ell)\leq  \beta < \ell+ n(1-\ell) + 2k(1-\ell)$:\\
Applying again Proposition \ref{Lq}  we
have
\[
  \|  Fu(t, \cdot)\|_{L^2}\lesssim \int_0^t (1+s) (1+t)^{(\ell-1)\frac{n}{2}} \left(\| |u(s, \cdot)|^p\|_{L^1}+ (1+s)^{(1-\ell)n\left( 1-\frac{1}{m}\right)}\| |u(s, \cdot)|^p\|_{L^m}\right) ds.\]
    Now we use the fractional Sobolev embedding
 \begin{eqnarray*}
\| u(s, \cdot)\|_{L^q} \lesssim  \| u(s, \cdot)\|_{\dot H^{ k(q)}}, \quad   k(q)= n\left(\frac12-\frac1q\right), \quad  2\leq q<\infty,
\end{eqnarray*}
by taking $q=p$ and $q=mp$, where $\frac{2n}{n+2} < m \leq 2$. On the one hand, if $ \beta > \ell+ n(1-\ell) + 2k_2(1-\ell)$, with
$k_2 =k(mp)=n\left(\frac12-\frac1{mp}\right)$
we may estimate
\begin{eqnarray*}
\| |u(s, \cdot)|^p\|_{L^{\frac{q}{p}}}& =&\| u(s, \cdot)\|_{L^{q}}^p\lesssim    \| u(s, \cdot)\|_{\dot H^{ k(q)}}^p  \\
&\lesssim& \| u(s, \cdot)\|_{L^2}^{(1-\theta)p} \| u(s, \cdot)\|_{\dot H^{\bar k}}^{p\theta}\lesssim
(1+s)^{(\ell-1)\left(\frac{n}{2} + k(q) \right)p}(\ln(e+s))^{\frac{p\theta}2}
\|u\|_{X(T)}^p\\
&\lesssim&(1+s)^{n(\ell-1)\left(p-\frac{p}{q}\right)}(\ln(e+s))^{\frac{p \tilde k}{2\bar k}}
\|u\|_{X(T)}^p,
\end{eqnarray*}
for $q=p$ and $q=mp$,  with $\theta \bar k= k(q)$ and    $p\geq  2$.
Hence, as before we conclude
\[ \| Fu(t, \cdot)\|_{L^2}\lesssim(1+t)^{\frac{n}{2}(\ell-1)}\| u\|_{X(T)}^p,\]
    for   $p> 1+ \frac{2}{n(1-\ell)}$.  On the other hand,     if
$1<  \beta \leq \ell+ n(1-\ell) + 2k_2(1-\ell)$ we may estimate
\begin{eqnarray*}
\| |u(s, \cdot)|^p\|_{L^1}& =&\| u(s, \cdot)\|_{L^{p}}^p\lesssim  \| u(s, \cdot)\|_{\dot H^{ k_1}}^p  \\
&\lesssim&
(1+s)^{p\max\{(\ell-1)\left(\frac{n}{2}+k_1\right), \frac{\ell-\beta}{2}\}} (\ln(e+s))^{\frac{p}2}
\|u\|_{X(T)}^p,
\end{eqnarray*}
with $ k_1=k(p)= n\left(\frac12-\frac1p\right)$, whereas
\begin{eqnarray*}
\| |u(s, \cdot)|^p\|_{L^m}& =&\| u(s, \cdot)\|_{L^{mp}}^p\lesssim    \| u(s, \cdot)\|_{\dot H^{ k_2}}^p  \\
&\lesssim& (1+s)^{\frac{\ell-\beta}{2} p} (\ln(e+s))^{\frac{p}2}
\|u\|_{X(T)}^p,
\end{eqnarray*}
with   $k_2=n\left(\frac12-\frac1{mp}\right)$.
Therefore
   \begin{eqnarray*}
  \| Fu(t, \cdot)\|_{L^2}
 & \lesssim &(1+t)^{\frac{n}{2}(\ell-1)}\int_0^t (1+s)^{1 +  \max\{n(\ell-1)(p-1), \frac{(\ell-\beta)p}{2}\}}(\ln(e+s))^{\frac{p}{2}}  ds\| u\|_{X(T)}^p\\  &+ &
   (1+t)^{\frac{n}{2}(\ell-1)}\int_0^t(1+s)^{1+(1-\ell)n\left( 1-\frac{1}{m}\right)+
    \frac{\ell-\beta}{2} p}(\ln(e+s))^{\frac{p}{2}}  ds\| u\|_{X(T)}^p\\
 &\lesssim &(1+t)^{\frac{n}{2}(\ell-1)}\| u\|_{X(T)}^p,\end{eqnarray*}
  for   $p> 1+ \frac{2}{n(1-\ell)}$, $m>\frac{2n}{n+2} $  and
  \[\beta> \ell+ n(1-\ell)+  \frac{2n(1-\ell)\ell}{2+ n(1-\ell)}.\]
Moreover, applying  again Proposition \ref{Lq} for $1 < \beta <  \ell+ n(1-\ell) + 2k(1-\ell)$,  we
have
\[
  \|  Fu(t, \cdot)\|_{\dot H^{k}}\lesssim \int_0^t (1+t)^{\frac{\ell-\beta}{2}} (1+s)^{1+\frac{\beta-\ell}{2}+(\ell-1) \left( \frac{n}{2}+k \right) } \left(\| |u(s, \cdot)|^p\|_{L^1}+ (1+s)^{(1-\ell)\left( \frac{n}{2}+k-1\right)}\| |u(s, \cdot)|^p\|_{\dot H^{k-1}}\right) ds.
\]
As before we may estimate
\begin{eqnarray*}
\| |u(s, \cdot)|^p\|_{L^1}& =&\| u(s, \cdot)\|_{L^{p}}^p\lesssim  \| u(s, \cdot)\|_{\dot H^{ k_1}}^p  \\
&\lesssim&  \|u\|_{X(T)}^p    \left\{ \begin{array}{lll}
(1+s)^{(\ell-1)n(p-1)} (\ln(e+s))^{\frac{p}2} \, \, \mbox{if $k_1 \leq \bar k$ } \\
(1+s)^{\frac{(\ell-\beta)p}{2}} \, \, \mbox{if $k_1 > \bar k$ }
\end{array}
\right.
\end{eqnarray*}
with $ k_1= n\left(\frac12-\frac1p\right) $,
then
$$
 \int_0^t (1+s)^{1+\frac{\beta-\ell}{2}+(\ell-1) \left( \frac{n}{2}+k \right) }  \| |u(s, \cdot)|^p\|_{L^1} \, ds \lesssim \| u\|_{X(T)}^p
$$
for   $p> 1+ \frac{2}{n(1-\ell)}$ and $\ell + \frac{4n(1-\ell)}{2+n(1-\ell) } < \beta <  \ell+ n(1-\ell) + 2k(1-\ell)$.\\
If  $p > \max\{1, k-1\}$ we may estimate
 \begin{eqnarray*}
 \| |u(s, \cdot)|^p\|_{{\dot H}^{k-1}}&\lesssim& \| u(s, \cdot)\|_{{\dot H}^{k-1}}\|u(s, \cdot)\|_{L^\infty}^{p-1}\\
&\lesssim&  \|u\|_{X(T)}^p  (1+s)^{\frac{(\ell-\beta)(p-1)}{2}}  \left\{ \begin{array}{lll}
(1+s)^{(\ell-1)\left( \frac{n}{2}+k-1\right)} (\ln(e+s))^{\frac{p}2} \, \, \mbox{if $k-1 \leq \bar k$ } \\
(1+s)^{\frac{\ell-\beta}{2}} \, \, \mbox{if $k-1 > \bar k$ }
\end{array}
\right.
 \end{eqnarray*}
The condition $k-1 \leq \bar k$ is equivalent to $\beta \geq  \ell+ n(1-\ell) + 2(k-1)(1-\ell)$.\\
For  $\ell > 1-\frac{2}{n}$, $k\geq 1+\frac{n\ell}{2}$ and $\beta \geq  \ell+ n(1-\ell) + 2(k-1)(1-\ell)$  we obtain
\[ \int_0^t  (1+s)^{1+\frac{\beta-\ell}{2}+(\ell-1) \left( \frac{n}{2}+k \right) } (1+s)^{(1-\ell)\left( \frac{n}{2}+k-1\right)} \| |u(s, \cdot)|^p\|_{\dot H^{k-1}} ds\lesssim \| u\|_{X(T)}^p\]
for   $p> 1+ \frac{2}{n(1-\ell)}$,
hence
\[
  \|  Fu(t, \cdot)\|_{\dot H^{k}}\lesssim (1+t)^{\frac{\ell-\beta}{2}} \| u\|_{X(T)}^p.
\]
Here we remark that
\[ \ell+ n(1-\ell) + 2(k-1)(1-\ell) \geq \ell + n(1-\ell) + \frac{2n(k(\ell-1)+1)(1-\ell)}{2-n(1-\ell)} \]
  for  all $k\geq 1+\frac{n\ell}{2}$.\\

\item The case $\bar k=k$, i.e., $\beta =  \ell+ n(1-\ell) + 2k(1-\ell)$:\\
In this case   one may conclude that
\[ \| Fu(t, \cdot)\|_{L^2}\lesssim(1+t)^{\frac{n}{2}(\ell-1)}\| u\|_{X(T)}^p,\]
   and
         \[
  \|  Fu(t, \cdot)\|_{\dot H^{\bar k}}\lesssim  (1+t)^{\frac{\ell-\beta}{2}}(\ln(e+t))^{\frac{1}2}\| u\|_{X(T)}^p
   \]
    for  $p> 1+ \frac{2}{n(1-\ell)}$.
     \end{itemize}

Finally, let us discuss the proof of \eqref{eq:contraction} only in the case
    $\beta> \ell+ n(1-\ell) + 2k(1-\ell)$.
Applying Proposition \ref{Lq}  we
have
\begin{eqnarray*}
\|  Fu(t, \cdot)-Fv(t, \cdot)\|_{L^2}&\lesssim & (1+t)^{(\ell-1)\frac{n}{2}}\int_0^t (1+s) \|(f(u)-f(v))(s, \cdot)\|_{L^1}ds\\
  &+& (1+t)^{(\ell-1)\frac{n}{2}} \int_0^t(1+s)^{1+(1-\ell)n\left( 1-\frac{1}{m}\right)}\|(f(u)-f(v))(s, \cdot)\|_{L^{m}} ds.
  \end{eqnarray*}
   Here,  we may take $m\in [1,2]$ such that $ m>\frac{2n}{n+2}$.\\
   By  using~\eqref{eq:f} and~H\"older inequality, we find that
\begin{equation}\label{eq:Holder}\begin{split}
& \|(f(u)-f(v))(s,\cdot)\|_{L^{\alpha}}\\
    & \qquad \leq C_1\,\|(u-v)(|u|^{p-1}+|v|^{p-1})(s,\cdot)\|_{L^{\alpha}} \\
    & \qquad \leq C_1\,\|(u-v)(s,\cdot)\|_{L^{p\alpha}}\,\big(\|u(s,\cdot)\|_{L^{p\alpha}}^{p-1}+\|v(s,\cdot)\|_{L^{p\alpha}}^{p-1}\big)\\
    & \qquad \leq C_2\,(1+s)^{n(\ell-1)\left(p-\frac1\alpha\right)}\,\|u-v\|_{X(T)}\,\big(\|u\|_{X(T)}^{p-1}+\|v\|_{X(T)}^{p-1}\big),
\end{split}\end{equation}
for any $1 \leq \alpha\leq m$. Therefore
\begin{eqnarray*}
\|  Fu(t, \cdot)-Fv(t, \cdot)\|_{L^2}&\lesssim & (1+t)^{(\ell-1)\frac{n}{2}}\int_0^t (1+s)^{1+n(\ell-1)\left(p-1\right)}ds\|u-v\|_{X(T)}\,\big(\|u\|_{X(T)}^{p-1}+\|v\|_{X(T)}^{p-1}\big)\\
  &+& (1+t)^{(\ell-1)\frac{n}{2}}  \|u-v\|_{X(T)}\,\big(\|u\|_{X(T)}^{p-1}+\|v\|_{X(T)}^{p-1}\big),
  \end{eqnarray*}
 for  $p> 1+ \frac{2}{n(1-\ell)}$.\\
 Applying again Proposition \ref{Lq} we have
  \begin{eqnarray*}
  \|  Fu(t, \cdot)-Fv(t, \cdot)\|_{\dot H^{k}}&\lesssim & (1+t)^{(\ell-1)\left(\frac{n}{2}+k\right)}\int_0^t (1+s) \| (f(u)-f(v))(s, \cdot)\|_{L^1}ds\\
  &+ &  (1+t)^{(\ell-1)\left(\frac{n}{2}+k\right)}\int_0^t(1+s)^{1+(1-\ell)\left( \frac{n}{2}+k-1\right)}\| (f(u)-f(v))(s, \cdot)\|_{\dot H^{k-1}} ds.
  \end{eqnarray*}
From now we assume that $f(u)=|u|^p$, without lose of generality.  In order to estimate $\|  (f(u)-f(v))(s, \cdot)\|_{\dot{H}^{k-1}}$ we use
\[  |u(s, x)|^{p}-|v(s, x)|^{p}= p\int_0^1 |v + \tau
(u-v)|^{p-2} (v + \tau (u-v))(s,x) d\tau (u-v)(s,x).
\]
Hence, applying  Proposition \ref{fractionalLeibniz}
 gives
\begin{eqnarray*}
\| |u(s, x)|^{p}-|v(s, x)|^{p}\|\|_{\dot{H}^{k-1}}&\lesssim & \|  (u-v)(s, \cdot)\|_{\dot{H}^{k-1}} \int_0^1 \| |v + \tau
(u-v)|^{p-2} (v + \tau (u-v))(s, \cdot)\|_{\infty} d\tau\\
&+&\|  (u-v)(s, \cdot)\|_{\infty} \int_0^1 \| |v + \tau
(u-v)|^{p-2} (v + \tau (u-v))(s, \cdot)\|_{\dot{H}^{k-1}} d\tau.
\end{eqnarray*}
 Now,  since $u,v\in X(T)$ we have
 $$
\|  (u-v)(s, \cdot)\|_{{\dot H}^{k-1}} \lesssim (1+s)^{(\ell-1)\left(\frac{n}{2}+k-1\right)}\|u-v\|_{X(T)},
$$
Applying Lemma \ref{lem:Sobolev}, for  $ \tilde k <\frac{n}{2}<k$ it follows
\[ \|  (u-v)(s, \cdot)\|_{\infty}\lesssim \|(u-v)(s, \cdot)\|_{\dot{H}^{\tilde k}} + \| (u-v)(s, \cdot)\|_{\dot{H}^{k}} \lesssim (1+s)^{(\ell-1)\left(\frac{n}{2}+\tilde k\right)}
\|u-v\|_{X(T)},\]
and
\[ \| |v + \tau
(u-v)|^{p-2} (v + \tau (u-v))(s, \cdot)\|_{\infty} \lesssim (1+s)^{(\ell-1)\left(\frac{n}{2}+\tilde k\right)(p-1)}\big(\|u\|_{X(T)}^{p-1}+\|v\|_{X(T)}^{p-1}\big),\]
  with $\tilde k=\frac{n}{2}-\epsilon_0$ and $\epsilon_0$ sufficiently small.\\
For $p > k$  Corollary \ref{corollaryfractionalpowers} implies
 \begin{eqnarray*}
&& \| |v + \tau
(u-v)|^{p-2} (v + \tau (u-v))(s, \cdot)\|_{\dot{H}^{k-1}}\leq C \| (v + \tau (u-v))(s, \cdot)\|_{{\dot H}^{k-1}}\|(v + \tau (u-v))(s, \cdot)\|_{L^\infty}^{p-2}\\
 &&\lesssim (1+s)^{(\ell-1)\left(\frac{n}{2}+k-1\right)}(1+s)^{(\ell-1)\left(\frac{n}{2}+\tilde k\right)(p-2)}\big(\|u\|_{X(T)}^{p-1}+\|v\|_{X(T)}^{p-1}\big).
\end{eqnarray*}
Therefore
  \begin{eqnarray*}
 && \|  Fu(t, \cdot)-Fv(t, \cdot)\|_{\dot H^{k}}\lesssim    (1+t)^{(\ell-1)\left(\frac{n}{2}+k\right)}\int_0^t (1+s)^{1+n(\ell-1)\left(p-1\right)}ds\|u-v\|_{X(T)}\,\big(\|u\|_{X(T)}^{p-1}+\|v\|_{X(T)}^{p-1}\big)\\
  &&+   (1+t)^{(\ell-1)\left(\frac{n}{2}+k\right)}\int_0^t(1+s)^{(\ell-1)\left(n-\epsilon_0\right)(p-1)} ds\|u-v\|_{X(T)}\,\big(\|u\|_{X(T)}^{p-1}+\|v\|_{X(T)}^{p-1}\big)\\
   &&\leq   (1+t)^{(\ell-1)\left(\frac{n}{2}+k\right)}\|u-v\|_{X(T)}\,\big(\|u\|_{X(T)}^{p-1}+\|v\|_{X(T)}^{p-1}\big),
  \end{eqnarray*}
 for  $p> 1+ \frac{2}{n(1-\ell)}$.

\end{proof}

\subsection{The proof of Theorem \ref{highdimensionssharp} }

By applying the change of variable
\[v(\tau, x)=u(t, x), \qquad 1+ \tau= \frac{(1+t)^{1-\ell}}{1-\ell},\]
the Cauchy problem (\ref{eq:DPE}) takes the form
\begin{equation}
\label{eq:DPE1}
\begin{cases}
 v_{\tau\tau}- \Delta v + \frac{\mu}{1+\tau} v_{\tau} =g(v), & \tau\geq s , \, x\in  \mathbf{R}^n, \,\\
 v(s, x)=0,  & x\in \mathbf{R}^n, \,\\
  v_{\tau}(s, x)=u_1(x),  & x\in \mathbf{R}^n,
\end{cases}
\end{equation}
with $ s=\frac{\ell}{1-\ell}$, $g(v)=[(1-\ell)(1+\tau)]^{\frac{2\ell}{1-\ell}}|v|^p$ and
\[ \mu=\frac{ \beta-\ell}{1-\ell}.\]

We now enunciate Corollary 2 from D'Abbicco's paper \cite{DA20}, which will be useful  in the proof of the Theorem 2.1.
There was introduced the following notation: For any $1\leq r\leq q\leq \infty$, let be
\[ d(r,q)=\left\{
\begin{array}{ll}
\frac{n}{r}-\frac{n-1}{2}-\frac1q, & \mbox{if  $r \leq q^{'}$} \\
 \frac{1}{r}+\frac{n-1}{2}-\frac{n}q, & \mbox{if $ r \geq q^{'}$}
\end{array}
\right.\]

\begin{corollary}(see \cite{DA20})
\label{Abbicco}
Let $ \mu \geq 2$. Let $n=2$ and $ 2 < q  \leq q_{\sharp}$, or $n=3$ and $q\in (1,4]$ or $n \geq 4$ and $\displaystyle{\frac{2(n-1)}{n+1} \leq q \leq q_{\sharp}}$.
Then there exists $r_2 \in (1,\min\{q, q^{\prime}\})$ such that $d(r_2,q)=1$ and the solution to (\ref{eq:DPE1})  verifies the following $(L^1 \cap L^2) - L^q$ decay estimate
$$
||  v(t, \cdot) ||_{L^q} \lesssim
(1+s)(1+\tau)^{-n\left(1-\frac{1}{q} \right)} \left ( || u_1||_{L^1} + (1+s)^{\frac{n-1}{2}- \frac{1}{q}} || u_1||_{L^{r_2}} \right )
$$
if $\displaystyle{\mu > n+1- \frac{2}{q}}$, and for any $\epsilon >0$ verifies the $(L^1 \cap L^2) - L^q$ estimate
$$
||  v(t, \cdot) ||_{L^q} \lesssim (1+s)^{\frac{\mu}{2}-\epsilon} (1+\tau)^{\epsilon-(n-1) \left ( \frac{1}{2}  -\frac{1}{q} \right ) - \frac{\mu}{2} }
 \left ( (1+s)^{\frac{1}{q}- \frac{n-1}{2}} || u_1||_{L^1}  + || u_1||_{L^{r_2}} \right )
$$
if $\displaystyle{\mu \leq n+1- \frac{2}{q}}$.
\end{corollary}

\begin{proof}(Theorem \ref{highdimensionssharp})
  We  define the space
\[
\label{eq:Xsp} X(T)
     \doteq  C([0, \infty),
L^{p_c}(\mathbf{R}^n)\cap L^{q_\sharp}(\mathbf{R}^n)),
\]
equipped with the norm
\begin{eqnarray*}
\label{xnorm}
\|v\|_{X(T)}   & \doteq & \sup_{\tau\in[0,T]}  \Bigl\{ (1+\tau)^{n\left(1-\frac{1}{p_c}\right)}\| v(\tau,\cdot)\|_{L^{p_c}} \\
&+ & (1+\tau)^{n\left(1-\frac{1}{q_\sharp}\right)}\| v(\tau,\cdot)\|_{L^{q_{\sharp}}} \Bigr\}.
\end{eqnarray*}
for $   \mu > n+1-\frac{2}{q_\sharp} $  and,
\begin{eqnarray*}
\label{xnorm}
\|v\|_{X(T)}   & \doteq & \sup_{\tau\in[0,T]}  \Bigl\{ (1+\tau)^{n\left(1-\frac{1}{p_c}\right)}\| v(\tau,\cdot)\|_{L^{p_c}}+ (1+\tau)^{n\left(1-\frac{1}{\bar q}\right)}\| v(\tau,\cdot)\|_{L^{\bar q}} \\
&+ & (1+\tau)^{(n-1)\left(\frac12-\frac1{q_{\sharp}}\right)+\frac{ \mu}{2}-\epsilon}\| v(\tau,\cdot)\|_{L^{q_{\sharp}}} \Bigr\}.
\end{eqnarray*}
for $   n+1-\frac{2}{\bar q}<  \mu \leq n+1-\frac{2}{q_\sharp} $ and,
\begin{eqnarray*}
\label{xnorm}
\|v\|_{X(T)}   & \doteq & \sup_{\tau\in[0,T]}  \Bigl\{  (1+\tau)^{n\left(1-\frac{1}{p_c}\right)}\| v(\tau,\cdot)\|_{L^{p_c}}+(1+\tau)^{(n-1)\left(\frac12-\frac1{\bar q}\right)+\frac{ \mu}{2}-\epsilon}\| v(\tau,\cdot)\|_{L^{\bar q}} \\
&+ & (1+\tau)^{(n-1)\left(\frac12-\frac1{q_{\sharp}}\right)+\frac{ \mu}{2}-\epsilon}\| v(\tau,\cdot)\|_{L^{q_{\sharp}}} \Bigr\}.
\end{eqnarray*}
for $ \mu=  n+1-\frac{2}{\bar q}$.\\
In the following we only verify how to prove the global existence in time assuming that that there exists $\delta > 0$ such that
$$
u_1\in \mathcal{D} \doteq L^2(\mathbf{R}^n) \cap L^1(\mathbf{R}^n) ,  \qquad  || u_1||_\mathcal{D}\leq \delta.
$$
Thanks to Corollary \ref{Abbicco}, if $\beta\geq 2 $, then $v^{0} \in X(T)$ and it satisfies
%
\[
\|v^{0}\|_{X} \leq C\, \| u_1\|_{\mathcal{D}}.
\]
Let us prove (\ref{eq:well_th2}). Applying  Corollary \ref{Abbicco} we
have for $   \mu > n+1-\frac{2}{q_\sharp} $ and for all $p_c\leq q\leq q_\sharp$
\[
  \| Fv(\tau, \cdot)\|_{L^q}\lesssim \int_0^t (1+s)^{1+\frac{2\ell}{1-\ell}} (1+\tau)^{-n\left(1-\frac1q\right)} \left(\| |v(s, \cdot)|^p\|_{L^1}+ (1+s)^{\frac{n-1}{2}-\frac1q}\| |v(s, \cdot)|^p\|_{L^{r(q)}}\right) ds\]
 with $r(q) \in [1,2[$  given by $\frac{n}{r(q)}= \frac12+ \frac{n}{2}+ \frac{1}{q}$. Taking into
account that $u\in X(T)$,   thanks to $r(q_\sharp)p_c< q_\sharp$ for all $p_c\leq q\leq q_\sharp$,
  we may estimate for $p_c<p\leq  \frac{q_\sharp}{r(q_\sharp)}$
\begin{eqnarray*}
\| |v(s, \cdot)|^p\|_{L^r}& =&\| v(s, \cdot)\|_{L^{rp}}^p
 \lesssim(1+s)^{-n\left(1-\frac1{pr(q)}\right)p }\|v\|_{X(T)}^p\\
&\lesssim&(1+s)^{-n\left(p-\frac1{r(q)}\right)}
\|v\|_{X(T)}^p.
\end{eqnarray*}
 Therefore,  for  $   \mu > n+1-\frac{2}{q_\sharp} $ we have for all $p_c\leq q\leq q_\sharp$
  \begin{eqnarray*}
  \| Fv(t, \cdot)\|_{L^q}
 & \lesssim &(1+\tau)^{-n\left(1-\frac1q\right)}\int_0^{\tau} (1+s)^{1+\frac{2\ell}{1-\ell} - n(p-1)}  ds\| v\|_{X(T)}^p\\
  & + &(1+\tau)^{-n\left(1-\frac1q\right)}\int_0^{\tau} (1+s)^{1+\frac{2\ell}{1-\ell} - n\left(p-\frac1{r(q)}\right)} (1+s)^{\frac{n-1}{2}-\frac1q} ds\| v\|_{X(T)}^p\\
 &\lesssim &(1+\tau)^{-n\left(1-\frac1q\right)}\| v\|_{X(T)}^p,\end{eqnarray*}
  for    \[p>  1+\frac{2}{n(1-\ell)} .\]
For   $  n+1-\frac{2}{\bar q}< \mu \leq n+1-\frac{2}{q_\sharp}$,  applying again Corollary \ref{Abbicco}, for $p_c\leq q\leq \bar q$ we may estimate
 \[
  \| Fv(\tau, \cdot)\|_{L^{q}}\lesssim (1+\tau)^{-n\left(1-\frac1{q}\right)}\int_0^{\tau} (1+s)^{1+\frac{2\ell}{1-\ell}} \left( \| |v(s, \cdot)|^p\|_{L^1}+ (1+s)^{\frac{n-1}{2} -\frac1{q}}\| |v(s, \cdot)|^p\|_{L^{r(q)}}\right) ds\]
 with $\frac{n}{r(q)}= \frac12+ \frac{n}{2}+ \frac{1}{q}$.
 For $   n+1-\frac{2}{\bar q}<\mu \leq n+1-\frac{2}{q_\sharp} $, we may estimate
\begin{eqnarray*}
\| |v(s, \cdot)|^p\|_{L^1}& =&\| v(s, \cdot)\|_{L^{p}}^p\lesssim \| v(s, \cdot)\|_{L^{p_c}}^{(1-\theta)p} \| v(s, \cdot)\|_{L^{q_\sharp}}^{\theta p}  \\
&\lesssim&(1+s)^{-n\left(1-\frac1{p_c}\right)(1-\theta)p+\left(\epsilon -(n-1)\left(\frac12-\frac1{q_\sharp}\right)-\frac{\mu}{2}\right)p\theta }
\|v\|_{X(T)}^p\lesssim(1+s)^{-n\left(1-\frac1{p_c}\right)p}
\|v\|_{X(T)}^p,
\end{eqnarray*}
  thanks to
\[ \epsilon -(n-1)\left(\frac12-\frac1{q_\sharp}\right)-\frac{\mu}{2}+n\left(1-\frac1{p_c}\right)\leq  \epsilon+ (n-1)\left(\frac1{q_\sharp}-\frac1{p_c}\right)\leq 0,\]
for $\epsilon>0$ and $\theta=\left(\frac1{p_c}- \frac1p\right)/ \left(\frac1{p_c}-\frac1{q_\sharp}\right)$.\\
Thanks to    $r(q)p_c\leq r(\bar q)p_c\leq  \bar q$ for all $p_c\leq q \leq \bar q$ and, for $   n+1-\frac{2}{\bar q}<\mu \leq n+1-\frac{2}{q_\sharp} $ we may estimate
\begin{eqnarray*}
\| |v(s, \cdot)|^p\|_{L^{r(q)}}& =&\| v(s, \cdot)\|_{L^{r(q)p}}^p\lesssim \| v(s, \cdot)\|_{L^{r(q)p_c}}^{(1-\theta)p} \| v(s, \cdot)\|_{L^{r(\bar q)p}}^{\theta p}  \\
&\lesssim&(1+s)^{-n\left(1-\frac1{r(q)p_c}\right)(1-\theta)p+\left(\epsilon -(n-1)\left(\frac12-\frac1{r(\bar q)p}\right)-\frac{\mu}{2}\right)p\theta }
\|v\|_{X(T)}^p,
\end{eqnarray*}
with $\theta=\left(\frac1{r(q)p_c}- \frac1{r(q)p}\right)/ \left(\frac1{r(q)p_c}-\frac1{r(\bar q)p}\right)$.
Now, 
\begin{eqnarray*} \gamma&=&-n\left(1-\frac1{r(q)p_c}\right)p+\left(\epsilon -(n-1)\left(\frac12-\frac1{r(\bar q)p}\right)-\frac{\mu}{2}+n\left(1-\frac1{r(q)p_c}\right)\right)p\theta\\
&\leq& -np + \frac{n}{r(q)} + \left(\epsilon -\frac{(n-1)}{2} -\frac1{r(\bar q)p}-\frac{n+1}{2}+\frac1{\bar q}+n\right)p\theta\\
&\leq& -np+\frac{n}{r(q)}+ \left( \frac1{r(\bar q)p_c}-\frac1{r(\bar q)p}\right)p\theta+ \epsilon p\theta\\
&\leq& -np+\frac{n}{r(q)}+\left(\frac1{r(\bar q)}\left(\frac1{p_c}-\frac1{p}\right)\right)p + \epsilon p\theta \\
&\leq&   -np\left(1-\frac1{p_c}\right)-\frac{np}{p_c} +\frac{n}{r(q)}+\frac{p}{p_c}-1 + \epsilon p\theta.\\
&\leq&   -np\left(1-\frac1{p_c}\right)-n\left(1-\frac{1}{r(q)}\right) +(n-1)\left(1-\frac{p}{p_c}\right)  + \epsilon p\theta.
\end{eqnarray*}
        Therefore, for $   n+1-\frac{2}{\bar q}<\mu \leq n+1-\frac{2}{q_\sharp} $  and for $p_c \leq q \leq \bar q$ we conclude that
   \begin{eqnarray*}
  \| Fv(t, \cdot)\|_{L^{q}}
 & \lesssim &(1+\tau)^{-n\left(1-\frac1{q}\right)}\int_0^{\tau} (1+s)^{1+\frac{2\ell}{1-\ell} - n\left(1-\frac1{p_c}\right)p}  ds\| v\|_{X(T)}^p\\
  & + &(1+\tau)^{-n\left(1-\frac1{q}\right)}\int_0^{\tau} (1+s)^{1+\frac{2\ell}{1-\ell}+ \frac{n-1}{2} -\frac1{q} +\gamma}  ds\| v\|_{X(T)}^p\\
  & + &(1+\tau)^{-n\left(1-\frac1{q}\right)}\int_0^{\tau} (1+s)^{1+\frac{2\ell}{1-\ell}- n\left(1-\frac1{p_c}\right)p  + \epsilon p\theta}  ds\| v\|_{X(T)}^p\\
  &\lesssim &(1+\tau)^{-n\left(1-\frac1{q}\right)}\| v\|_{X(T)}^p,\end{eqnarray*}
  for    \[p>  \frac{2}{n(1-\ell)}\frac{p_c}{p_c-1}=1+\frac{2}{n(1-\ell)} .\]
   Now, for   $ \mu= n+1-\frac{2}{\bar q}$,  applying again  Corollary \ref{Abbicco}, we may estimate
$\| Fv(\tau, \cdot)\|_{L^{p_c}}$ as before, whereas for $q=\bar q$ or $q=q_\sharp$
\begin{eqnarray*}
  \| Fv(\tau, \cdot)\|_{L^{q}}&\lesssim& (1+\tau)^{\epsilon -n\left(1-\frac1{q}\right)-\frac{\mu}{2}}\int_0^{\tau} (1+s)^{\frac{2\ell}{1-\ell}+\frac{ \mu}{2}-\epsilon} \left( (1+s)^{-\frac{n-1}{2} +\frac1{q}}\| |v(s, \cdot)|^p\|_{L^1}+ \| |v(s, \cdot)|^p\|_{L^{r(q)}}\right) ds\\
  &\lesssim& (1+\tau)^{\epsilon -n\left(1-\frac1{q}\right)-\frac{\mu}{2}}\int_0^{\tau} (1+s)^{\frac{2\ell}{1-\ell}+1-\epsilon}\| |v(s, \cdot)|^p\|_{L^1}+(1+s)^{\frac{2\ell}{1-\ell}+\frac{ \mu}{2}-\epsilon} \| |v(s, \cdot)|^p\|_{L^{r(q)}} ds,
  \end{eqnarray*}
 for any $\epsilon>0$, with $\frac{n}{r(q)}= \frac12+ \frac{n}{2}+ \frac{1}{q}$.    \\
     Taking into
account that $u\in X(T)$, as before, we may estimate
 \[\| |v(s, \cdot)|^p\|_{L^1}\lesssim(1+s)^{-n\left(1-\frac1{p_c}\right)p}
\|v\|_{X(T)}^p\]
and thanks to $r(\bar q)p_c\leq r(q_\sharp)p_c < q_\sharp$(see Remark \ref{crucial}), we may estimate for  $p_c< p\leq \frac{q_\sharp}{r(q_\sharp)}$
 \begin{eqnarray*}
\| |v(s, \cdot)|^p\|_{L^{r(q)}}& =&\| v(s, \cdot)\|_{L^{pr(q)}}^p\\
&\lesssim &
 (1+s)^{\left(\epsilon -(n-1)\left(\frac12-\frac1{pr(q)}\right)-\frac{\mu}{2}\right)p }
\|v\|_{X(T)}^p
\end{eqnarray*}
 for any $\epsilon>0$.\\
  Now we may write
    \[\frac{2\ell}{1-\ell}+\frac{ \mu}{2}-\epsilon +\left(\epsilon -(n-1)\left(\frac12-\frac1{r(q)p}\right)-\frac{\mu}{2}\right)p=
    1+\frac{2\ell}{1-\ell}+\epsilon(p-1)+ n- \frac{1}{r(\bar q)} - \left(n-1+\mu \right)\frac{p}{2} + \gamma, \]
  with
$$
\gamma =   \frac{ \mu}{2}-1+
n\left(\frac{1}{r(q)}-1\right)+\frac{1}{r(\bar q)}- \frac{1}{r(q)}.
$$
For $ \mu= n+1-\frac{2}{\bar q}$ and $q=\bar q$ we have that $\gamma=0$, whereas for
 $q=q_\sharp$ we have
\begin{eqnarray*}
  \gamma &= &  \frac{ \mu}{2}-1+ n\left(\frac{1}{r(q_\sharp)}-1\right)+\frac{1}{r(\bar q)}- \frac{1}{r(q_\sharp)} \\
   &=& \frac{(n+1)}{2} - \frac{1}{\bar q} -1 + \frac{(1-n)}{2} +
   \frac{1}{q_\sharp} + \frac{1}{n}  \left (  \frac{1}{\bar q} -
\frac{1}{q_\sharp} \right) \\
&=& \left ( \frac{1}{n} - 1\right ) \left (  \frac{1}{\bar q} -
\frac{1}{q_\sharp} \right) < 0
\end{eqnarray*}
We conclude that for $ \mu= n+1-\frac{2}{\bar q}$ and $q=\bar q$ or
$q=q_\sharp$
  \begin{eqnarray*}
  \| Fv(\tau, \cdot)\|_{L^{q}}
 & \lesssim &
   (1+\tau)^{\epsilon -(n-1)\left(\frac12-\frac1{q}\right)-\frac{ \mu}{2}}\int_0^{\tau}
   (1+s)^{1+\frac{2\ell}{1-\ell}+\epsilon(p-1)+ n- \frac{1}{r(\bar q)} - \left(n-1+\mu \right)\frac{p}{2} + \gamma} ds\\
    & \lesssim &
   (1+\tau)^{\epsilon -(n-1)\left(\frac12-\frac1{q}\right)-\frac{ \mu}{2}}\| v\|_{X(T)}^p,
  \end{eqnarray*}
  for any $\epsilon>0$, $p>  p_c=1+\frac{2}{n(1-\ell)}$ and
$$
1+\frac{2\ell}{1-\ell}+\epsilon(p-1)+ n- \frac{1}{r(\bar q)} - \left(n-1+\mu
\right)\frac{p}{2} < -1
$$
i.e.
$$
(n-1+\mu) \frac{p_c}{2} \geq \frac{2}{1-\ell} + n- \frac{1}{r(\bar q)}
$$
is equivalent to
$$
\mu \geq n+1 - \frac{2}{p_cr(\bar q)}=n+1 - \frac{2}{\bar q}.
$$

Moreover, for   $  n+1-\frac{2}{\bar q}< \mu\leq n+1-\frac{2}{q_\sharp}$, we have
\begin{small}
  \begin{eqnarray*}
  \| Fv(\tau, \cdot)\|_{L^{q_\sharp}}&\lesssim& (1+\tau)^{\epsilon -(n-1)\!\left(\frac12-\frac1{q_\sharp}\right)-\frac{ \mu}{2}}\int_0^{\tau} (1+s)^{\frac{2\ell}{1-\ell}+\frac{ \mu}{2}-\epsilon} \left( (1+s)^{-\frac{n-1}{2} +\frac1{q_\sharp}}\| |v(s, \cdot)|^p\|_{L^1}+ \| |v(s, \cdot)|^p\|_{L^{r(q_\sharp)}}\!\right) ds\\
  &\lesssim& (1+\tau)^{\epsilon -(n-1)\left(\frac12-\frac1{q_\sharp}\right)-\frac{ \mu}{2}}\int_0^{\tau} (1+s)^{\frac{2\ell}{1-\ell}+1-\epsilon} \| |v(s, \cdot)|^p\|_{L^1}+ (1+s)^{\frac{2\ell}{1-\ell}+\frac{ \mu}{2}-\epsilon}\| |v(s, \cdot)|^p\|_{L^{r(q_\sharp)}} ds
  \end{eqnarray*}
  \end{small}
 for any $\epsilon>0$, with $\frac{n}{r(q_\sharp)}= \frac12+ \frac{n}{2}+ \frac{1}{q_\sharp}$.\\
 If  $  n+1-\frac{2}{p_cr(q_\sharp)}< \mu\leq n+1-\frac{2}{q_\sharp}$
 we may estimate
 \begin{eqnarray*}
\| |v(s, \cdot)|^p\|_{L^{r(q_\sharp)}}\lesssim(1+s)^{\left(\epsilon -(n-1)\left(\frac12-\frac1{pr(q_\sharp)}\right)-\frac{\mu}{2}\right)p }
\|v\|_{X(T)}^p,
\end{eqnarray*}
hence
\begin{eqnarray*}
(1+s)^{\frac{2\ell}{1-\ell}+\frac{ \mu}{2}-\epsilon}\| |v(s, \cdot)|^p\|_{L^{r(q_\sharp)}}&\leq &(1+s)^{\frac{2\ell}{1-\ell}+ \epsilon(p-1)
-\frac{ \mu(p-1)}{2} -(n-1)\left(\frac12-\frac1{pr(q_\sharp)}\right)p}\\
&\leq &  (1+s)^{\frac{2\ell}{1-\ell} -n(p-1) +1+ \epsilon(p-1)+\gamma} \leq (1+s)^{-1}
\end{eqnarray*}
for   $\epsilon>0$ sufficiently small and
\[p_c< p\leq \left(\frac1{r(q_\sharp)}-\frac1{q_\sharp} \right)p_cr(q_\sharp)=p_c+1-\frac{p_cr(q_\sharp)}{q_\sharp}\]
thanks to
\[\gamma=\frac{p-1}{p_cr(q_\sharp)}+\frac1{q_\sharp}-\frac1{r(q_\sharp)}\leq 0.\]
Finally,  if  $  n+1-\frac{2}{\bar q}< \mu\leq n+1-\frac{2}{p_cr(q_\sharp)}$ we may estimate for $p_c< p\leq \frac{q_\sharp}{r(q_\sharp)}$
 \begin{eqnarray*}
\| |v(s, \cdot)|^p\|_{L^{r(q_\sharp)}}& =&\| v(s, \cdot)\|_{L^{pr(q_\sharp)}}^p\\
&\lesssim &
 (1+s)^{\left(\epsilon -(n-1)\left(\frac12-\frac1{pr(q_\sharp)}\right)-\frac{\mu}{2}\right)p }
\|v\|_{X(T)}^p
\end{eqnarray*}
 for any $\epsilon>0$.\\
  Now we may write
    \[\frac{2\ell}{1-\ell}+\frac{ \mu}{2}-\epsilon +\left(\epsilon -(n-1)\left(\frac12-\frac1{r(q_\sharp)p}\right)-\frac{\mu}{2}\right)p=
    1+\frac{2\ell}{1-\ell}+\epsilon(p-1)+ n- \frac{1}{r(\bar q)} - \left(n-1+\mu \right)\frac{p}{2} + \gamma, \]
  with
$$
\gamma =   \frac{ \mu}{2}-1+
n\left(\frac{1}{r(q_\sharp)}-1\right)+\frac{1}{r(\bar q)}- \frac{1}{r(q_\sharp)}.
$$
It remains to prove that $\gamma\leq 0$. Indeed,  for $ \mu\leq n+1-\frac{2}{p_cr(q_\sharp)}$ and for $\frac{1}{q_{\sharp}}\leq \frac{1}{n-1} \left (  \frac{n}{p_cr(q_\sharp)} - \frac{1}{\bar q} \right )$ (see Remark \ref{aboutqbar}) we have
\begin{eqnarray*}
  \gamma &= &  \frac{ \mu}{2}-1+ n\left(\frac{1}{r(q_{\sharp})}-1\right)+\frac{1}{r(\bar q)}- \frac{1}{r(q_{\sharp})} \\
&= &  \frac{ \mu}{2}-1+  \frac{1-n}{2}+\frac{1}{q_{\sharp}} +     \frac{1}{n} \left ( \frac{1}{\bar q}- \frac{1}{q_{\sharp}} \right )  \\
& \leq &  \frac{ \mu}{2}-  \frac{1+n}{2}+ \frac{1}{p_cr(q_\sharp)}  \leq 0
\end{eqnarray*}
Therefore,  if  $  n+1-\frac{2}{\bar q}< \mu\leq n+1-\frac{2}{q_\sharp}$ we have proved that
\[\| Fv(s, \cdot)\|_{L^{q_\sharp}}\lesssim(1+\tau)^{\epsilon -(n-1)\left(\frac12-\frac1{q_\sharp}\right)-\frac{ \mu}{2}}\| v\|_{X(T)}^p,\]
 for any $\epsilon>0$ sufficiently small, $p>  p_c=1+\frac{2}{n(1-\ell)}$.
\end{proof}

%

%

\section*{Appendix}
\label{SecAppendix}
%
%
In the Appendix we list some notations used through the paper and
results of Harmonic Analysis which are important tools for proving
results on the global existence of small data solutions for
semi-linear  models with power non-linearities. Through this paper,
we use the following.\\
For~$s\geq0$, we denote by~$|D|^{s}f=\mathcal F^{-1}(|\xi|^{s}\hat f)$ and $\langle D \rangle^{s}f=\mathcal F^{-1}(\langle \xi \rangle^{s}\hat f)$, with
$\langle \xi \rangle^s= (1+ |\xi|^2)^{\frac{s}{2}}$.

For any $q\in[1,\infty]$, we denote by $L^q(\mathbf{R}^n)$ the usual
Lebesgue space over $\mathbf{R}^n$.
Let $s\in \mathbb{R}$ and $1<p<\infty$. Then
\begin{align*}
H^{s, p}(\mathbb{R}^n)&=\{u\in\mathcal{S}'(\mathbb{R}^n): \|\langle D \rangle^s u\|_{L^p(\mathbb{R}^n)}=\| u\|_{H^s_p(\mathbb{R}^n)}<\infty \},\\
\dot{H}^{s,p}(\mathbb{R}^n)&=\{u\in\mathcal{Z}'(\mathbb{R}^n): \||D|^s u\|_{L^p(\mathbb{R}^n)}=\| u\|_{\dot{H}^s_p(\mathbb{R}^n)}<\infty \}
\end{align*}
are called Bessel and Riesz potential spaces, respectively. If $p=2$, then we use the notations $H^s(\mathbb{R}^n)$ and
 $\dot{H}^s(\mathbb{R}^n)$, respectively. In the definition of the Riesz potential spaces we use the space of distributions $\mathcal{Z}'(\mathbb{R}^n)$.
This space of distributions can be identified with the factor space $\mathcal{S}'/\mathcal{P}$, where $\mathcal{S}'$ denotes the dual of Schwartz space and $\mathcal{P}$ denotes the set of all polynomials.\\
We recall that~$H^{s,q}(\mathbf{R}^n)=W^{s,q}(\mathbf{R}^n)$, the
usual Sobolev space, for any~$q\in(1,\infty)$ and~$s\in \mathbf{N}$.\\
The following inequality can be found in \cite{Friedman}, Part 1,
Theorem 9.3.
\begin{proposition}[Fractional Gagliardo-Nirenberg inequality] \label{GagliardoNirenberginequality}
Let $1<p,p_0,p_1<\infty$, $\sigma >0$ and $s\in [0,\sigma)$. Then it holds the following fractional Gagliardo-Nirenberg inequality for all $u\in L^{p_0}(\mathbb{R}^n)\cap \dot{H}^{\sigma, p_1}(\mathbb{R}^n)$:
\[
\|u\|_{\dot{H}^{s, p}}\lesssim \|u\|_{L^{p_0}}^{1-\theta}\|u\|_{\dot{H}^{\sigma, p_1}}^\theta,
\]
where $\theta=\theta_{s,\sigma}(p,p_0,p_1)=\frac{\frac{1}{p_0}-\frac{1}{p}+\frac{s}{n}}{\frac{1}{p_0}-\frac{1}{p_1}+\frac{\sigma}{n}}$ and $\frac{s}{\sigma}\leq \theta\leq 1$ .
\end{proposition}
We present here a result for fractional powers  \cite{Sickel}.
\begin{proposition} \label{fractionalpowers} Let $p>1$, $f(u)=|u|^p$ or $f(u)=|u|^{p-1}u$ and $u \in H^{s,m}$, where $s \in \big(\dfrac{n}{m},p\big)$, $1<m<\infty$. Then the
following estimate holds$:$
\[\Vert f(u)\Vert_{H^{s,m}}\le C \Vert u\Vert_{H^{s,m}}\|u\|_{L^\infty}^{p-1}.\]
\end{proposition}
In  \cite{EbertReissig}  the following corollary was derived:
\begin{corollary} \label{corollaryfractionalpowers}
Let $f(u)=|u|^p$ or $f(u)=|u|^{p-1}u$, with $p>\max\{1, s\}$ and $u
\in H^{s,m}\cap L^\infty$, $1<m<\infty$. Then the following estimate
holds$:$
\[\Vert f(u)\Vert_{\dot{H}^{s,m}}\le C \Vert u\Vert_{\dot{H}^{s,m}}\|u\|_{L^\infty}^{p-1}.\]
\end{corollary}
We refer to \cite{DAEL17} for the nex result:
\begin{lemma}\label{lem:Sobolev}
Let $0<2s_1<n<2s_2$. Then for any function $f\in \dot{H}^{s_1}\cap \dot{H}^{s_2}$ one has
%
\[\| f\|_{\infty}\lesssim \| f\|_{\dot{H}^{s_1}} + \| f\|_{\dot{H}^{s_2}}.\]
%
\end{lemma}

The next result combine in some sense some familiar results as
Leibniz rule for the product of two function and H\"{o}lder's
inequality for derivatives of fractional order (Theorem 7.6.1 in \cite{Grafakos}):
\begin{proposition} \label{fractionalLeibniz}
Let us assume $s>0$ and $1\leq r \leq \infty, 1< p_1,p_2,q_1,q_2 \leq
\infty$ satisfying the relation \[
\frac{1}{r}=\frac{1}{p_1}+\frac{1}{p_2}=\frac{1}{q_1}+\frac{1}{q_2}.\]
Then the following fractional Leibniz rules hold:
\[
\|\,|D|^s(u \,v)\|_{L^r}\lesssim \|\,|D|^s
u\|_{L^{p_1}}\|v\|_{L^{p_2}}+\|u\|_{L^{q_1}}\|\,|D|^s v\|_{L^{q_2}}
\]
 for any $u\in \dot{H}^{s, p_1}(\mathbf{R}^n)\cap L^{q_1}(\mathbf{R}^n)$ and $v\in \dot{H}^{s, q_2}(\mathbf{R}^n)\cap L^{p_2}(\mathbf{R}^n)$,
\[
\|\langle D \rangle^s(u \,v)\|_{L^r}\lesssim \|\langle D \rangle^s
u\|_{L^{p_1}}\|v\|_{L^{p_2}}+\|u\|_{L^{q_1}}\|\langle D\rangle^s
v\|_{L^{q_2}}
\]
for any $u\in H^{s, p_1}(\mathbf{R}^n)\cap L^{q_1}(\mathbf{R}^n)$ and
$v\in H^{s, q_2}(\mathbf{R}^n)\cap L^{p_2}(\mathbf{R}^n)$.
\end{proposition}

\section*{Acknowledgments}

The first  author is partially supported by Fapesp grant number 2020/08276-9 and CNPq grant number 304408/2020-4.
The second author is partially supported by “FCT – Fundação para a Ciência e a Tecnologia, I.P., Project UIDB/05037/2020”.

\end{document}